\documentclass[11pt]{amsart}
\usepackage{amsmath}
\usepackage{amssymb}
\usepackage{amscd}

\def\NZQ{\mathbb}               
\def\NN{{\NZQ N}}
\def\QQ{{\NZQ Q}}
\def\ZZ{{\NZQ Z}}
\def\RR{{\NZQ R}}

\newtheorem{Theorem}{Theorem}[section]
\newtheorem{Lemma}[Theorem]{Lemma}

\newtheorem{Proposition}[Theorem]{Proposition}

\newtheorem{Example}[Theorem]{Example}

\newtheorem{Question}[Theorem]{Question}
\let\epsilon\varepsilon
\let\phi=\varphi
\let\kappa=\varkappa

\begin{document}

\title{Finite generation of extensions of associated graded rings along a valuation}
\author{Steven Dale Cutkosky}
\thanks{partially supported by NSF}

\address{Steven Dale Cutkosky, Department of Mathematics,
University of Missouri, Columbia, MO 65211, USA}
\email{cutkoskys@missouri.edu}


\begin{abstract}
In this paper we consider the  question of when the associated graded ring along a valuation, ${\rm gr}_{\nu^*}(S)$, is a finite ${\rm gr}_{\nu^*}(R)$-module, where $S$ is a normal local ring which lies over a normal local ring $R$ and $\nu^*$ is a valuation of the quotient field of $S$ which dominates $S$. 

We begin by discussing some examples and results allowing us to refine the conditions under which finite generation can hold. We must impose the condition that the extension of valuations is {\it defectless} and perform  a birational extension of $R$ along the valuation to obtain finite generation (replacing $S$ with the local ring of the quotient field  of $S$ determined by the valuation which lies over the extension of $R$). With these assumptions, we have that finite generation holds, when $R$ is a two dimensional excellent local ring. 

Our main result (in Theorem \ref{ThmAJV})  is to show that for an arbitrary valuation in an algebraic function field over an arbitrary  field of characteristic zero, after a birational extension along the valuation, we always have finite generation (all finite extensions of valued fields are defectless in characterisitic zero). This generalizes an earlier result,  in \cite{C11}, showing that finite generation holds (after a birational extension) with the additional assumptions that $\nu$ has rank 1 and has an algebraically closed residue field.   There are essential difficulties in removing these assumptions, which are addressed in the proof in this paper. 

  We    obtain general results for unramified extensions of excellent local rings  in Proposition \ref{PropUnR}, showing that after blowing up, the extension of associated graded rings is finitely generated of  an extremely  simple form.  This proposition plays an essential role  in  the proof of Theorem \ref{ThmAJV}.
\end{abstract}

\maketitle

\section{Introduction}

Suppose that $K$ is a field. Associated to a valuation $\nu$ of $K$  is a value group $\Phi_{\nu}$ and  a valuation ring $V_{\nu}$ with maximal ideal $m_{\nu}$. Let $R$ be a local domain with quotient field $K$. We say that $\nu$ dominates $R$ if $R\subset V_{\nu}$ and $m_{\nu}\cap R=m_R$ where $m_R$ is the maximal ideal of $R$. We have an associated semigroup $S^R(\nu)=\{\nu(f)\mid 0\ne  f\in R\}$, as well as the associated graded ring of $R$ along $\nu$
\begin{equation}\label{eqN31}
{\rm gr}_{\nu}(R)=\bigoplus_{\gamma\in \Phi_{\nu}}\mathcal P_{\gamma}(R)/\mathcal P^+_{\gamma}(R)=\bigoplus_{\gamma\in S^{R}(\nu)}\mathcal P_{\gamma}(R)/\mathcal P^+_{\gamma}(R)
\end{equation}
which is defined by Teissier in \cite{T1}. Here 
$$
\mathcal P_{\gamma}(R)=\{f\in R\mid \nu(f)\ge \gamma\}\mbox{ and }\mathcal P^+_{\gamma}(R)=\{f\in R\mid \nu(f)> \gamma\}.
$$ 
This ring plays an important role in local uniformization of singularities (\cite{T1} and \cite{T2}).
The ring ${\rm gr}_{\nu}(R)$ is a domain, but it is often not Noetherian, even when $R$ is.  In fact, a necessary condition for ${\rm gr}_{\nu}(R)$ to be Noetherian is that $\Phi_{\nu}$ be a finitely generated group. 

Some recent papers on valuation theory and  local uniformization  in positive characteristic are: Cossart and Piltant \cite{CP1} and \cite{CP2}, Ghezzi, H\'a and Kashcheyeva \cite{GHK}, Ghezzi and  Kashcheyeva \cite{GK}, Herrera Govantes, Olalla Acosta, Spivakovsky and Teissier \cite{GAST}, Knaf and Kuhlmann \cite{KK},  Kuhlmann \cite{Ku1}, \cite{Ku2} and \cite{Ku3},  Novacoski and Spivakovsky \cite{NS}, Spivakovsky \cite{Sp}, Teissier \cite{T1} and  \cite{T2}, Temkin \cite{Te} and Vaqui\'e \cite{Vaq}.
Some recent papers on resolution of singularities in positive characteristic are:  Benito and Villamayor \cite{BeV},  Bravo and Villamayor \cite{BrV},  Cossart, Jannsen and Saito \cite{CJS},  Hauser \cite{Ha} and Hironaka \cite{H1}.

Suppose that $K\rightarrow K^*$ is a finite extension of  fields and $\nu^*$ is a valuation which is an extension of $\nu$ to $K^*$. We have the classical indices
$$
e(\nu^*/\nu)=[\Phi_{\nu^*}:\Phi_{\nu}]\mbox{ and }f(\nu^*/\nu)=[V_{\nu^*}/m_{\nu^*}:V_{\nu}/m_{\nu}]
$$
as well as the defect  $\delta(\nu^*/\nu)$ of the extension. Ramification of valuations and the defect are discussed in Chapter VI of \cite{ZS2}, \cite{E} and Kuhlmann's papers \cite{Ku1} and \cite{Ku3}. A survey is given in Section 7.1 of \cite{CP}. By Ostrowski's lemma, if $\nu^*$ is the unique extension of $\nu$ to $K^*$, we have that
\begin{equation}\label{int3}
[K^*:K]=e(\nu^*/\nu)f(\nu^*/\nu)p^{\delta(\nu^*/\nu)}
\end{equation}
where $p$ is the characteristic of the residue field $V_{\nu}/m_{\nu}$. From this formula, the defect can be computed using Galois theory in an arbitrary finite extension. 
If  $V_{\nu}/m_{\nu}$ has characteristic 0, then $\delta(\nu^*/\nu)=0$ and  $p^{\delta(\nu^*/\nu)}=1$, so there is no defect. Further, if $\Phi_{\nu}=\ZZ$ and $K^*$ is separable over $K$ then there is no defect. 

Now suppose that $K\rightarrow K^*$ is a finite separable  field extension and $\nu^*$ is a valuation of $K^*$ with restriction $\nu$ to $K$. 
Suppose that $R$ and $S$ are normal, excellent local rings with  quotient field  ${\rm QF}(R)=K$ of $R$ and quotient field  ${\rm QF}(S)=K^*$ of $S$. Suppose that $S$ dominates $R$ and $\nu^*$ dominates $S$. We have the following basic question:

\begin{Question}\label{Q1}
When is ${\rm gr}_{\nu^*}(S)$ a finitely generated ${\rm gr}_{\nu}(R)$-algebra?
\end{Question}

We will see that Question \ref{Q1} has a good general answer, but we must refine the question, which we will do  by consideration of examples.

In Corollary 6 \cite{C13},  it is shown that if $K=K^*$, $R$ and $S$ are two dimensional excellent regular local rings with $R\ne S$, and $\nu$ is not discrete, then ${\rm gr}_{\nu}(S)$ is not a finitely generated ${\rm gr}_{\nu}(R)$-algebra.  Thus to obtain a good answer to Question \ref{Q1} for general valuations, we must restrict to the case that $S$ lies over $R$
($S$ is the local ring which is the localization of the integral closure of $R$ in $K^*$ that is dominated by $\nu^*$).

An algebraic local ring $A$ in an algebraic function field $L$ over a field $k$ is a local ring which is a localization of a finite type $k$-algebra such that ${\rm QF}(A)=L$.
We always assume that a valuation $\tilde\nu$ of an algebraic function field $L$ over a field $k$ is a $k$-valuation; that is, $\tilde\nu|(k\setminus 0)=0$.
Now it can happen, even when $S$ lies over $R$ and $R$ and $S$ are algebraic regular local rings in  two dimensional  algebraic function fields over an arbitrary field, that ${\rm gr}_{\nu^*}(S)$ is not a finitely generated ${\rm gr}_{\nu}(R)$-algebra (Example 9.4 \cite{CV1} and Example 1.2 \cite{C11}). 

However, if $R$ has dimension two and contains a field $k$ of characteristic zero, then there exists a regular local ring $R_1$ which dominates $R$ and is dominated by $\nu$ such that if $S_1$ is the normal local ring of $K^*$ which lies over $R_1$ and is dominated by $\nu^*$, then ${\rm gr}_{\nu^*}(S_1)$ is a finite ${\rm gr}_{\nu}(R_1)$-module.  Further, this property is stable under further blowing up (\cite{GHK}, \cite{CV2} and \cite{C12}).

Suppose that $S$ lies over $R$. Then it may be that ${\rm gr}_{\nu^*}(S)$ is not  integral over ${\rm gr}_{\nu}(R)$ (Example 1.1 \cite{C11}),  even when $K$ and $K^*$ are algebraic function fields over an arbitrary field $k$. However, after some blowing up $R\rightarrow R_1$ along $\nu$, we obtain that ${\rm gr}_{\nu^*}(S_1)$ is integral over ${\rm gr}_{\nu}(R_1)$, where $S_1$ is the normal algebraic local ring of $K^*$ which is dominated by $\nu^*$ and lies over $R_1$ (Theorem 1.4 \cite{C11}). This explains the finiteness of ${\rm gr}_{\nu^*}(S_1)$ over ${\rm gr}_{\nu}(R_1)$ in the two dimensional result cited above from \cite{GHK}, \cite{CV2}, \cite{C12}.

Suppose that $R_1$ is a local domain with quotient field $K$. We will say that $R_1$ is a birational extension of $R$ if $R_1$ dominates $R$ and $R_1$ is a localization of a finite type $R$-algebra.  This leads us to a refinement of Question \ref{Q1}:

\begin{Question}\label{Q2}
Does there exist a birational extension $R\rightarrow R_1$ such that $R_1$ is normal and $\nu$ dominates $R_1$ such that ${\rm gr}_{\nu^*}(S_1)$ is a finitely generated ${\rm gr}_{\nu}(R_1)$-module, where $S_1$ is the normal local ring with quotient field $K^*$ which lies over $R_1$ and is dominated by $\nu^*$?
\end{Question}

In Section 7.11 \cite{CP} (recalled in Example 1.3 \cite{C11})  an example is given in a separable extension $K\rightarrow K^*$ of two dimensional algebraic function fields over an algebraically closed field $k$ of positive characteristic $p$  such that ${\rm gr}_{\nu^*}(S_1)$ is not a finitely generated ${\rm gr}_{\nu}(R_1)$-algebra for all regular local rings $R_1$ birationally dominating $R$ which are dominated by $\nu$. This example has positive defect $\delta(\nu^*/\nu)>0$. This example is not sporadic, but illustrates a general principal in dimension two which we now state. 

\begin{Theorem}(Theorem 0.1 \cite{C13})
Suppose that $R$ is a 2 dimensional excellent local domain with quotient field $K$. Further suppose that $K^*$ is a finite separable extension of $K$ and $S$ is a 2 dimensional local domain with quotient field
$K^*$ such that  $S$ dominates $R$. 
Suppose that $\nu^*$ is a valuation of $K^*$ such that 
 $\nu^*$ dominates $S$. Let $\nu$ be the restriction of $\nu^*$ to $K$.  Then the extension $(K,\nu)\rightarrow (K^*,\nu^*)$ is without defect if and only if there exist regular local rings $R_1$ and $S_1$ such that
 $R_1$ is a local ring of a blow up of $R$, $S_1$ is a local ring of a blowup of $S$, $\nu^*$ dominates $S_1$, $S_1$ dominates $R_1$ and ${\rm gr}_{\nu^*}(S_1)$ is a finitely generated ${\rm gr}_{\nu}(R_1)$-algebra.
 \end{Theorem}
 
 Thus to obtain a good answer to Question \ref{Q2} for general valuations, we must assume that $K\rightarrow K^*$ is defectless ($\delta(\nu^*/\nu)=0$). This leads us to our final formulation of the question. 
 
 \begin{Question}\label{Q3}
 Suppose that $K\rightarrow K^*$ is defectless. Does there exist a birational extension $R\rightarrow R_1$ such that $R_1$ is normal and $\nu$ dominates $R_1$, and ${\rm gr}_{\nu^*}(S_1)$ is a finitely generated ${\rm gr}_{\nu}(R_1)$-module, where $S_1$ is the normal local ring with quotient field $K^*$ which lies over $R_1$ and is dominated by $\nu^*$?
 \end{Question}

Question \ref{Q3} has a positive answer when $R$ is assumed to have dimension two, as follows from results of \cite{GHK}, \cite{GK}, \cite{CP} and \cite{C12}.
In this paper we give a positive answer to Question \ref{Q3} in algebraic function fields of arbitrary dimension over an arbitrary field of characteristic zero. 
We prove the following theorem in Section \ref{SecAJ}.

\begin{Theorem}\label{ThmAJV} Suppose that $K$ is an algebraic function field over a field $k$ of characteristic zero and $K^*$ is a finite extension of $K$. Suppose that $\nu^*$ is a $k$-valuation of $K^*$ ($\nu^*(k\setminus 0)=0$). Let $\nu$ be the restriction of $\nu^*$ to $K$, and suppose that $\tilde R$ is an algebraic local ring of $K$ which is dominated by $\nu$. Then there exists an algebraic regular local ring $R$ of $K$ which dominates $\tilde R$ and is dominated by $\nu$ such that    if $S$ is the local ring of the integral closure of $R$ in $K^*$ which is dominated by  $\nu^*$,  then 
 ${\rm gr}_{\nu^*}(S)$ is a free ${\rm gr}_{\nu}(R)$-module of finite rank $ef$ where $e=[\Phi_{\nu^*}/\Phi_{\nu}]$ and $f= [V_{\nu^*}/m_{\nu^*}:V_{\nu}/m_{\nu}]$.
\end{Theorem}

We  deduce the following proposition from the proof of Theorem \ref{ThmAJV}..

\begin{Proposition}\label{PropAJV2} Let assumptions and conclusions be as in Theorem \ref{ThmAJV}. Further assume that $K$ contains an algebraically closed field $k'$ such that $k'\cong V_{\nu}/m_{\nu}$. Then
$\Phi_{\nu^*}/\Phi_{\nu}$ acts on ${\rm gr}_{\nu^*}(S)$ with ${\rm gr}_{\nu^*}(S)^{\Phi_{\nu^*}/\Phi_{\nu}}\cong {\rm gr}_{\nu}(R)$.
\end{Proposition}

Theorem \ref{ThmAJV} generalizes Theorem 1.6 \cite{C11}, which establishes Theorem \ref{ThmAJV}  with  the restrictions  that $\nu^*$ has rank 1, $k$ is algebraically closed  and $V_{\nu^*}/m_{\nu^*}\cong k$.
The proof in \cite{C11} relies on the fact that for a valuation $\nu$ of $K$ of  rank 1 dominating $R$, there is a natural extension $\hat\nu$ of $\nu$ to the quotient field of the completion $\hat R$ which dominates $\hat R$, such that the value group is not very different from that of $\nu$ (\cite{Sp}, \cite{CG}). After blowing up to obtain that $R$ is regular and the discriminant ideal of $R\rightarrow K^*$ is generated by a monomial in a regular system of parameters $x_1,\ldots,x_n$ in $R$, the Abhyankar Jung Theorem \cite{Ab2} gives an inclusion of $k$-algebras (assuming $k$ is algebraically closed of characteristic zero)
$$
\hat R=k[[x_1,\ldots,x_n]]\rightarrow \hat S\rightarrow k[[x_1^{\frac{1}{d}},\ldots,x_n^{\frac{1}{d}}]]
$$
such that $\hat S$ is the invariant ring of a subgroup of $\ZZ_d^n$. Then using the strong monomialization theorem of \cite{C} and \cite{CP}, after blowing up some more to obtain a monomial map which captures the invariants of the extension of valuations,  we show that there is a set of generators of $\hat S$ as an $\hat R$-module whose values are a complete set of representatives of the cosets of $\Phi_{\nu}$ in $\Phi_{\nu^*}$, from which the conclusions of Theorem \ref{ThmAJV} follow, in the case that $\nu$ has rank 1 and $k=V_{\nu^*}/m_{\nu^*}$ is algebraically closed (of characteristic zero).

The significant difficulty in extending this proof to the general case of the statement of Theorem \ref{ThmAJV} is when $\nu$ has arbitrary rank. In this case the structure of an extension $\hat\nu$ of $\nu$ to a valuation dominating $\hat R$ is not well understood, although the structure is known to be complicated (\cite{GAST}). In fact, even under a finite unramified  extension the semigroup of a valuation of rank $>1$ can increase, as shown in the example at the end of Section \ref{SecUnR}, although the value groups will stay the same. 

We prove the following proposition on the extension of associated graded rings under an unramified extension. Related problems are considered in \cite{GAST}.

\begin{Proposition}\label{PropUnR}  Suppose that $R$ and $S$ are  normal local rings such that $R$ is excellent,  $S$ lies over  $R$ and $S$ is unramified over $R$, $\tilde \nu$ is a valuation of the quotient field $L$ of $S$ which dominates $S$, and $\nu$ is the restriction of $\tilde \nu$ to the quotient field $K$ of $R$. Suppose that $L$ is finite separable over $K$.  Then there exists a normal local ring $R'$, which is a birational extension of $R$ and  is dominated by $\nu$  such that if $R''$ is a normal local ring which is a birational extension of $R'$ and is
dominated by $\nu$ and $S''$ is the normal local ring of $L$ which is dominated by $\tilde\nu$ and lies over $R''$, then $R''\rightarrow S''$ is unramified, and
$$
{\rm gr}_{\tilde\nu}(S'')\cong {\rm gr}_{\nu}(R'')\otimes_{R''/m_{R''}}S''/m_{S''}.
$$
\end{Proposition} 

As remarked above, we give an example at the end of 
 Section \ref{SecUnR} showing  that a birational extension of $R$ may be necessary to obtain the conclusions of Proposition \ref{PropUnR}.
 
 We use ramification theory and Proposition \ref{PropUnR} to reduce the proof of Theorem \ref{ThmAJV} to the case when $\nu^*$ is the unique extension of $\nu$ to $K^*$ and we have an equality of residue fields $V_{\nu^*}/m_{\nu^*}=V_{\nu}/m_{\nu}$. We derive in Theorem \ref{SSM} an extension of the strong monomialization theorem of \cite{C} and \cite{CP}, which allows us to find, after some blow ups along the valuation to capture in $R\rightarrow S$ invariants of the extension of valuations,  a set of generators of $S$ as an $R$-module whose values are a complete set of representatives of the cosets of $\Phi_{\nu}$ in $\Phi_{\nu^*}$, from which Theorem \ref{ThmAJV} follows. 

\section{notation} 
We will denote the
maximal ideal of a local ring $R$ by $m_R$.
We will denote the  quotient field of a domain $R$ by $QF(R)$.
Suppose that $R\subset S$ is an inclusion of local rings. We will say that $R$ dominates
$S$ if $m_S\cap R=m_R$.  
Suppose that $\tilde R$ is a local domain with quotient field $K={\rm QF}(R)$. We will say that $\tilde R$ is a birational extension of $R$ if $\tilde R$ dominates $R$ and $\tilde R$ is a localization of a finite type $R$-algebra.
Suppose that $K^*$ is a finite extension of a field $K$,
 $R$ is a local ring with quotient field $QF(R)=K$ and $S$ is a local ring with quotient field $QF(S)=K^*$. Suppose that $R$ and $S$ are normal. We will say
that $S$ lies over $R$ and $R$ lies below $S$ if $S$ is a localization at a maximal ideal of the integral closure of $R$ in $K^*$. 
If $R$ is a local ring, $\hat R$ will denote the completion of $R$ at its maximal ideal.
If $M$ is a finite field extension of a field $L$, we will denote the group of
$L$-automorphisms of $M$ by $\text{Aut}(M/L)$. If $M$ is  Galois extension of $L$ we will write $G(M/L)=\mbox{Aut}(M/L)$.

Good introductions to the valuation theory which we require in this paper can be found  in Chapter VI of \cite{ZS2} and in \cite{RTM}.  
Let $\nu$ be a valuation of a field $K$.
We will denote by $V_{\nu}$ the associated valuation ring, and the maximal ideal of $V_{\nu}$ by $m_{\nu}$. 
The value group of a valuation $\nu$ will be denoted by $\Phi_{\nu}$.  If $R$ is a subring of $V_{\nu}$ then the center of $\nu$ (the center of $V_{\nu}$)
on $R$ is the prime ideal $R\cap m_{\nu}$. 

Let ${\rm gr}_{\nu}(R)$ be the associated graded ring of $R$ along $\nu$ defined in (\ref{eqN31}). For $f\in R$, let the initial form ${\rm in}_{\nu}(f)\in {\rm gr}_{\nu}(R)$ be the class of $f$ in $\mathcal P_{\nu(f)}(R)/\mathcal P^+_{\nu(f)}(R)$.

Suppose that $R$ is a local domain. A monoidal transform $R\rightarrow R_1$ is a 
birational extension of local domains such that $R_1=R[\frac{P}{x}]_m$ where $P$ is
a  prime ideal of $R$ such that $R/P$ is regular, $0\ne x\in P$ and $m$ is a prime ideal of $R[\frac{P}{x}]$
such that $m\cap R=m_R$. $R\rightarrow R_1$ is called a quadratic transform  if $P=m_R$.

If $R$ is regular, and $R\rightarrow R_1$ is a monoidal transform, 
then there exists a regular system of parameters $(x_1,\ldots, x_n)$ in
$R$ and $r\le n$ such that
$$
R_1=R\left[\frac{x_2}{x_1},\ldots,\frac{x_r}{x_1}\right]_m.
$$

Suppose that $\nu$ is a valuation of the quotient field of $R$ with valuation ring $V_{\nu}$
which dominates $R$. Then $R\rightarrow R_1$ is a monoidal transform along $\nu$
(along $V_{\nu}$) if $\nu$ dominates $R_1$.

We will use the following properties of an excellent ring $R$ (from Scholie IV.7.8.3 \cite{EGAIV}). If $R'$ is a localization of a finite type $R$-algebra then $R'$ is excellent. If $R$ is a domain and $R'$ is the integral closure of $R$ in a finite field extension of the quotient field of $R$ then $R'$ is a finite $R$-module. If $R$ is a normal local ring, then its completion $\hat R$ is a normal local ring.

If $\Lambda$ is a subset of a group $H$, then we will denote the group generated by $\Lambda$ by $G(\Lambda)$.

\section{Associated graded rings in splitting fields and inertia fields}\label{SecVT}

We now introduce some notation which we will use throughout this section.  We refer to Sections 10 and 11 of Chapter VI \cite{ZS2}.
Suppose that $K$ is a field with a valuation $\nu$, and that $\nu$ dominates an excellent normal local ring $R$ which has $K$ as its quotient field.  Let $V_{\nu}$ be the valuation ring of $\nu$ with maximal ideal $m_{\nu}$.  Let $r$ be the rank of $\nu$, which is finite since $\nu$ dominates the Noetherian local ring $R$ (by Proposition 1, page 330 \cite{ZS2}). Let
\begin{equation}\label{eqVT1}
0=\Phi_r\subset \Phi_{r-1}\subset \cdots \subset \Phi_1\subset \Phi_0=\Phi_{\nu}
\end{equation}
be the chain of isolated subgroups of the value group $\Phi_{\nu}$ of $\nu$. Let
\begin{equation}\label{eqVT2}
0=I_0\subset I_1\subset \cdots\subset I_r=m_{\nu}
\end{equation}
be the chain of prime ideals in $V_{\nu}$. The $\Phi_i$ are related to the $I_i$ as follows. Let 
$$
U_i=\{\nu(a)\mid 0\ne a\in I_i\}.
$$
 Then 
\begin{equation}\label{eq60}
\mbox{$\Phi_i$ is the complement of $U_i$ and $-U_i$ in $\Phi_{\nu}$.}
\end{equation}
For $1\le i\le r$, let $\nu_i$ be the specializations of $\nu$; the valuation ring of $\nu_i$ on $A$ is $V_{\nu_i}=(V_{\nu})_{I_i}$ and the value group of $\nu_i$ is $\Phi_0/\Phi_i$. We  have that $\nu_r=\nu$.

If $A$ is a subring of $V_{\nu}$, then the centers of the  specializations of $V_{\nu}$ on $A$ is the chain of prime ideals
$$
0=I_0\cap A\subset I_1\cap A\subset \cdots\subset I_r\cap A=m_{\nu}\cap A.
$$

Let $K^*$ be a finite extension field of $K$, and let $\nu^*$ be an extension of $\nu$ to $K^*$.  Let $\nu_1^*=\nu^*, \nu_2^*,\ldots,\nu^*_a$ be  all of the extensions of $\nu$ to $K^*$. Let
$$
0=I_{i.0}^*\subset I_{i,1}^*\subset \cdots \subset I_{i,r}^*=m_{\nu^*_i}
$$
be the chain of prime ideals in $V_{\nu^*_i}$ for $1\le i\le a$.  Let $\nu_{j,i}^*$ be the specializations of $V_{\nu_j^*}$, with valuation rings $V_{\nu_{j,i}^*}=(V_{\nu_j^*})_{I^*_{j,i}}$. Let
$$
0=\Phi_r^*\subset \cdots \subset \Phi_1^*\subset \Phi_0^*=\Phi_{\nu^*}
$$
be the isolated subgroups of the value group $\Phi_{\nu^*}$ of $\nu^*$.
The value group of 
$V_{\nu_{1,i}^*}$ is $\Phi_0^*/\Phi_i^*$.

Define a chain of prime ideals in $R$ by
\begin{equation}\label{eq1}
0=P_0=I_0\cap R\subset P_1=I_1\cap R\subset \cdots\subset P_r=I_r\cap R=m_R.
\end{equation}

We have that $P_j=I_{i,j}^*\cap R$ for $1\le i\le a$ and $0\le j\le r$. 

\begin{Lemma}\label{Lemma10} There exists a normal local ring $R'$ which birationally dominates $R$, such that $\nu$ dominates $R'$, and if $R''$ is a normal local ring which birationally dominates $R'$ and is dominated by $\nu$, then the prime ideals $I_i\cap R''$ are all distinct, and 
$$
{\rm trdeg}_{{\rm QF}(R''/R''\cap I_i)}{\rm QF}(V_{\nu}/I_i)=0
$$
for all $i$.
\end{Lemma} 

\begin{proof} Let $I_{r+1}=V_{\nu}$. We have that ${\rm trdeg}_{{\rm QF}(R/R\cap I_i)}{\rm QF}(V_{\nu}/I_i)<\infty$
 for $0\le i\le r$  by Proposition 1, page 330 \cite{ZS2}. Thus there exist $z_{1,j},\ldots,z_{\alpha_j,j}\in I_j$ for $1\le j\le r+1$ such that for $0\le i\le r$, 
$$
z_{1,i+1},\ldots,z_{\alpha_{i+1},i+1}
$$ 
is a transcendence basis of $(V_{\nu}/I_i)_{I_i}$ over ${\rm QF}(R/R\cap I_i)$. Let $A$ be the integral closure of 
$$
R[z_{i,j}|1\le j\le r+1, 1\le i\le \alpha_j]
$$
 in $K$, and let $R'=A_{A\cap m_{\nu}}$. Then $R'$ satisfies the conclusions of the lemma.
\end{proof}

After replacing $R$ with $R'$, we may assume that $R=R'$ satisfies the conclusions of Lemma \ref{Lemma10}.

 Let $T$ be the integral closure of $R$ in $K^*$ and let  $m_j=T\cap m_{\nu_j^*}$ for $1\le j\le a$ and $m=m_1$.  The ring $T$ is a finite $R$-module since $R$ is excellent.

\begin{Lemma}\label{Crit} There exists a birational extension $R'$ of $R$, where $R'$ is a normal local ring which is dominated by $\nu$, such that
if $R''$ is a normal local ring which is a birational extension of $R'$ which is dominated by $\nu$ and $V_{\nu_{j,i}^*}\ne V_{\nu_{k,i}^*}$ for some $i,j,k$, then 
$I_{j,i}^*\cap C\ne I_{k,i}^*\cap C$ and $I_{j,i}^*\cap C\not\subset I_{k,r}^*\cap C$ where $C$ is the integral closure of $R''$ in $K^*$.
\end{Lemma}

\begin{proof} Let $A=\cap_{j=1}^a V_{\nu_j^*}$ be the integral closure of of $V_{\nu}$ in $K^*$, so that $A_{A\cap I_{j,i}^*}=V_{\nu_{j,i}^*}$ for all $i,j$ (by Proposition 2.36 \cite{RTM} and Lemma 2.37 \cite{RTM}). We will first show that 
\begin{equation}\label{eqN1}
I_{j,i}^*\cap A\ne I_{k,i}^*\cap A\mbox{ implies }I_{j,i}^*\cap A\not\subset I_{k,r}^*\cap A.
\end{equation}
Suppose that $P=A\cap I_{j,i}^*\subset M=A\cap I_{k,r}^*$.  By Lemma 2.17 \cite{RTM}, $A_P=V_{\nu_{j,i}^*}$ and $A_M=V_{\nu^*_k}$, so $V_{\nu_{j,i}^*}$ is a specialization (localization) of $V_{\nu_k^*}$. Since $V_{\nu_{j,i}^*}$ has dimension  $i$, we have that $V_{\nu_{j,i}^*}=V_{\nu_{k,i}^*}$ and thus $I_{j,i}^*\cap A=I_{k,i}^*\cap A$.

By (\ref{eqN1}), 
\begin{equation}\label{eqN2}
\mbox{there exist $u_{k,j,i}\in A\cap I_{j,i}^*$ such that $u_{k,j,i}\not\in A\cap I_{k,r}^*$ if $A\cap I_{j,i}^*\ne A\cap I_{k,i}^*$.}
\end{equation}

There exist relations
\begin{equation}\label{eqN3}
u_{k,j,i}^{c_{k,j,i}}+\alpha_{k,j,i,c_{k,j,i}-1}u_{k,j,i}^{c_{k,j,i}-1}+\cdots+\alpha_{j,k,i,0}=0
\end{equation}
with $\alpha_{k,j,i,l}\in V_{\nu}$ for $0\le l\le c_{k,j,i}-1$. Let $B$ be the integral closure of
$$
R[\{\alpha_{k,j,i,l}\mid 1\le i\le r, 0\le l\le c_{k,j,i}-1\}]
$$
in $K$ and let $R'=B_{m_{\nu}\cap B}$.

Suppose that $R''$ is a birational extension of $R'$ which is normal and is dominated by $\nu$. Let $C$ be the integral closure of $R''$ in $K^*$. 
Suppose that $V_{\nu_{j,i}^*}\ne V_{\nu_{k,i}^*}$. Then $I_{j,i}^*\cap A\ne I_{k,i}^*\cap A$. Thus $u_{k,j,i}\in I_{j,i}^*\cap C$ is such that $u_{k,j,i}\not\in I_{k,r}^*\cap C$ by (\ref{eqN2}).
Since $I_{k,i}^*\cap C\subset I_{k,r}^*\cap C$, we further have that $u_{k,j,i}\not\in I_{k,i}^*\cap C$.
\end{proof}

From now on in this section, assume that $K^*$ is a Galois extension of $K$.    Decomposition and inertia groups are defined and analyzed in Section 12, Chapter VI of \cite{ZS2} and in Sections 7 and 11 of \cite{RTM}.

\begin{Lemma}\label{Lemma2}  We have that the decomposition groups
$G^s(V_{\nu_j^*}/V_{\nu})\subset G^s(T_{m_j}/R)$ and 
$$
G^s(V_{\nu_j^*}/V_{\nu})=G^s(T_{m_j}/R)
$$
 if and only if $\nu_j^*$ is the unique extension of $\nu$ to $K^*$ which dominates $T_{m_j}$.
\end{Lemma}

\begin{proof}  Certainly $G^s(V_{\nu_j^*}/V_{\nu})\subset G^s(T_{m_j}/R)$ since $m_{\nu_j^*}\cap T_{m_j}=m_{T_{m_j}}$. Suppose that $\nu'$ is an extension of $\nu$ to $K^*$ which dominates $T_{m_j}$. Then there exists $\sigma\in G(K^*/K)$ such that $\sigma(V_{\nu_j^*})=V_{\nu'}$, and we have that $\sigma(T_{m_j})=T_{m_j}$, so $\sigma\in G^s(T_{m_j}/R)$. Thus $G^s(V_{\nu_j^*}/V_{\nu})=G^s(T_{m_j}/R)$ if and only if $\nu_j^*$ is the unique extension of $\nu$ to $K^*$ which dominates $T_{m_j}$.
\end{proof}

Let $T^s$ be the integral closure of $R$ in $K^s=K^{G^s(V_{\nu^*}/V_{\nu})}$ and let $n_j=m_{\nu_j^*}\cap T^s$ for $1\le j\le a$, $n=n_1$. Let $\nu_j^s$ be the restriction of $\nu_j^*$ to $K^s$ for $1\le j\le a$, $\nu^s=\nu^s_1$.

We  recall a technique to compute norms and traces. Let $K'$ be an intermediate field of the Galois extension $K\rightarrow K^*$. Let $T'$ be the integral closure of $R$ in $K'$. Let $G=G(K^*/K)$ be the Galois group of $K^*$ over $K$ and $G'=G(K^*/K')$ be the Galois group of $K^*$ over $K'$. Let $\sigma_1,\ldots,\sigma_n$ be a complete set of representatives of the cosets of $G'$ in $G$. Then the norm and trace of an element $x\in T'$ of $K'$ over $K$ can be computed  (by formulas (19) and (20) on page 91 \cite{ZS1}) as
$$
{\rm N}(x)={\rm N}_{K'/K}(x)=\prod_{i=1}^n\sigma_i(x)
$$
and
$$
{\rm Tr}(x)={\rm Tr}_{K'/K}(x)=\sum_{i=1}^n\sigma_i(x).
$$

We have that ${\rm N}_{K'/K}(x), {\rm Tr}_{K'/K}(x)\in R$ by formulas (8) and (9) on page 88 and Theorem 4, page 260 \cite{ZS1}.

 Let $G=G(K^*/K)$ be the Galois group of $K^*$ over $K$ and  let
$$
G_{j,i}=G^s(V_{\nu_{j,i}^*}/V_{\nu_i})
$$
for $1\le i\le r$. For fixed $j$, we have inclusions
$$
(1)\subset G_{j,r}=G^s(V_{\nu_j^*}/V_{\nu})\subset G_{j,r-1}\subset \cdots \subset G_{j,2}\subset G_{j,1}\subset G
$$
with a tower of fixed fields
\begin{equation}\label{eq16}
K\subset K^{j,1}\subset K^{j,2}\subset \cdots\subset K^{j,r-1}\subset K^{j,r}\subset K^*,
\end{equation}
where $K^{j,i}=K^{G_{j,i}}$. Let $G_i=G_{1,i}$ for $1\le i\le r$.

\begin{Lemma}\label{Lemma8} Let $R'$ be the birational extension of $R$ of the conclusions of Lemma \ref{Crit}.
If $R''$ is a normal local ring which is a birational extension of $R'$ which is dominated by $\nu$, then 
$$
G_{j,i}=G^s(V_{\nu_{j,i}^*}/V_{\nu_i})=G^s(C_{I_{j,i}^*\cap C}/R''_{I_i\cap R''})
$$
for  $1\le j\le a$ and $1\le i\le r$, where $C$ is the integral closure of $R''$ in $K^*$.
\end{Lemma}

\begin{proof} By Lemma \ref{Crit}, $I_{j,i}^*\cap C=I_{k,i}^*\cap C$ implies $V_{\nu_{j,i}^*}=V_{\nu_{k,i}^*}$.
Suppose $\sigma\in G(K^*/K)$. Then 
$\sigma(I_{j,i}^*\cap C)=I_{j,i}^*\cap C$ implies $\sigma(I_{j,i}^*)=I_{j,i}^*$,
 so that
$$
G^s(C_{C\cap I_{j,i}^*}/(R'')_{R''\cap I_i})\subset G^s(V_{\nu_{j,i}^*}/V_{\nu_i}).
$$
Now
$$
G^s(V_{\nu_{j,i}^*}/V_{\nu_i})\subset G^s(C_{C\cap I_{j,i}^*}/(R'')_{R''\cap I_i})
$$
since $V_{\nu_{j,i}^*}$ dominates $C_{C\cap I_{j,i}^*}$ and $V_{\nu_i}$ dominates $(R'')_{R''\cap I_i}$.
\end{proof}

After replacing $R$ with $R'$, we may assume that $R=R'$ satisfies the conclusions of Lemma \ref{Crit}. 
 Let $S_i$ be a subset of $\{1,2,\ldots, a\}$ such that $1\in S_i$, if $1\le k\le a$ then $V_{\nu^*_{k,i}}=V_{\nu^*_{j,i}}$ for some $j\in S_i$  and $V_{\nu^*_{j,i}}\ne V_{\nu^*_{k,i}}$ if $j\ne k$ are in $S_i$.

Then
\begin{equation}\label{eq12}
G_{j,i}=G^s(T_{Q_{j,i}}/R_{P_i})
\end{equation}
for $j\in S_i$ and $1\le i\le r$, where $Q_{j,i}=T\cap m_{\nu^*_{j,i}}$ for $j\in S_i$.

\begin{Lemma}\label{Lemma6} For fixed $i$ with $1\le i\le r$, the prime ideals $\{Q^s_{j,i}\}_{j\in S_i}$ in $T^s$ are pairwise coprime, where $Q^s_{j,i}=Q_{j,i}\cap T^s$.
\end{Lemma}

\begin{proof} Suppose that $j,k\in S_i$ with $j\ne k$ and $Q_{j,i}^s$ and $Q_{k,i}^s$ are not coprime. Then there exists a maximal ideal $Q_{l,r}^s$ of $T^s$ such that $Q_{j,i}^s\subset Q_{l,r}^s$ and $Q_{k,i}^s\subset Q_{l,r}^s$. There exists a maximal ideal $Q_{n,r}$ of $T$ such that $Q_{n,r}\cap T^s=Q_{l,r}$ (by Lemma 1.20, page 13 \cite{RTM}) and there exist prime ideals $Q_{a,i}\subset Q_{n,r}$ and $Q_{b,i}\subset Q_{n,r}$ such that $Q_{a,i}\cap T^s=Q_{j,i}^s$ and $Q_{b,i}\cap T^s=Q_{k,i}^s$ by the going down theorem (Proposition 1.24B, page 15 \cite{RTM}). But $Q_{a,i}$ and $Q_{b,i}$ are necessarily distinct, giving a contradiction to the conclusions of Lemma \ref{Crit}.
\end{proof}

For $0\le i\le r-1$, let $\sigma_1^i={\rm id}, \sigma_1^i,\ldots,\sigma_{t_i}^i\in G_i$ be  a complete set of representatives of the cosets of $G_{i+1}$ in $G_i$. Then
$$
\sigma_{i_0}^0\sigma_{i_1}^1\cdots \sigma_{i_{r-1}}^{r-1},\,\,1\le i_j\le t_j\mbox{ and }0\le j\le r-1
$$
is a complete set of representatives of the cosets of $G_r=G^s(V_{\nu^*}/V_{\nu})$ in $G$.

Suppose $i$ is fixed with $1\le i\le r$. By Lemma \ref{Lemma6} and the Chinese remainder theorem, given $n_i\in \NN$, there exists $y_i\in T^s$ such that
\begin{equation}\label{eq7}
y_i\equiv 1\mod (Q_{1,i}^s)^{n_i}\mbox{ and }y_i\equiv 0 \mod (Q_{j,i}^s)^{n_i}\mbox{ if $j\in S_i$ and }j>1.
\end{equation}

By equation (\ref{eq12}), (\ref{eq16})  and  Proposition 1.46 \cite{RTM}, $Q_{1,i}$ is the unique prime  ideal of $T$ lying over $Q_{1,i}^s$ for $0\le i\le r$. Thus 
$$
y_i\equiv 1 \mod Q_{1,i}^{n_i}\mbox{ and }y_i\equiv 0\mod Q_{j,i}^{n_i}\mbox{ if $j\in S_i$ and }j>1.
$$

\begin{Lemma}\label{Lemma7} Suppose that 
$$
\tau=\sigma_1^0\sigma_1^1\cdots\sigma_1^{k-1}\sigma_{i_k}^k\cdots \sigma_{i_{r-1}}^{r-1}
$$
with $i_k>1$. Then
$$
\tau(y_j)\equiv 1\mod (Q_{1,j})^{n_j}\mbox{ if }k\ge j
$$
and
$$
\tau(y_j)\equiv 0 \mod (Q_{1,j})^{n_j}\mbox{ if }k\le j-1.
$$
\end{Lemma}

\begin{proof}  If $k\ge j$, then $\tau\in G_k\subset G_j$ so $\tau(y_j)\equiv 1\mod (Q_{1,j})^{n_j}\mbox{ if }k\ge j$.

Now suppose that $k\le j-1$. If $\sigma\in G_a$ then $\sigma$ permutes the $Q_{i,b}$ with $a\le b$ which  contain $Q_{1,a}$, and permutes the $Q_{i,b}$ with $a\le b$ which do  not contain $Q_{i,a}$. We have that 
$\sigma_{i_{k+1}}^{k+1}\cdots\sigma_{i_{r-1}}^{r-1}\in G_{k+1}$ so 
$$
\sigma_{i_{k+1}}^{k+1}\cdots\sigma_{i_{r-1}}^{r-1}(y_j)\equiv 0\mod Q_{i,j}^{n_j}\mbox{ if } Q_{1,k+1}\not\subset Q_{i,j}.
$$
Since $\sigma_{i_k}^{k}(Q_{1,k+1})\ne Q_{1,k+1}$, we have that 
$$
\sigma_{i_k}^k\sigma_{i_{k+1}}^{k+1}\cdots\sigma_{i_{r-1}}^{r-1}(y_j)\equiv 0\mod Q_{1,j}^{n_j}.
$$
Thus $\tau(y_j)\equiv 0\mod (Q_{1,j})^{n_j}$.
\end{proof}

An element  $\sigma\in G$ is in $G_i$ if and only if the conjugate valuation $\nu^*_{1,i}\sigma=\nu^*_{1,i}$ (page 68, \cite{ZS2}). 
Thus for $0\ne f\in K^*$ and $\sigma\in G_i$, 
$$
\nu^*(\sigma(f))-\nu^*(f)\in \Phi_i^*.
$$
Since $Q_{1,i}=I_{1,i}^*\cap T$, and since the conclusions of Lemma \ref{Lemma10} are assumed to hold for $R=R'$, we have by (\ref{eq60}) that
$$
\nu^*(Q_{1,i})\in \Phi_{i-1}^*\setminus \Phi_i^*\mbox{ for }1\le i\le r,
$$
where 
$$
\nu^*(Q_{1,i})=\min\{\nu^*(g)\mid g\in Q_{1,i}\}.
$$
Suppose that $0\ne f\in T^s$.  
\begin{equation}\label{eq9}
\nu^*(y_1\cdots y_rf)=\nu^*(y_1)+\cdots+\nu^*(y_r)+\nu^*(f)=\nu^*(f)
\end{equation}
by (\ref{eq7}) and since $Q_{1,j}\subset Q_{1,r}$ for all $j$.
For  $0\le i\le r-1$, let
$$
\phi_i(f)=\min\{\nu^*(\sigma(f))-\nu^*(f)\mid \sigma\in G_i\}.
$$
We have that $\phi_i(f)\in \Phi_i^*$.
Let $n_i\in\NN$ be such that 
\begin{equation}\label{eq8}
n_{i+1}\nu^*(Q_{1,i+1})>-\phi_i(f)\mbox{ for }0\le i\le r-1.
\end{equation}
Suppose that $\tau$ is as in the assumptions of Lemma \ref{Lemma7} with $\tau\ne {\rm id}$. Then
\begin{equation}\label{eq10}
\begin{array}{lll}
\nu^*(\tau(y_1\cdots y_r f))-\nu^*(f)&=& [\nu^*(\tau(f))-\nu^*(f)]+\nu^*(\tau(y_1))+\cdots+\nu^*(\tau(y_r))\\
&\ge& \phi_k(f)+\nu^*(\tau(y_{k+1}))+\cdots+\nu^*(\tau(y_r))\\
&\ge& \phi_k(f)+n_{k+1}\nu^*(Q_{1,k+1})+\cdots+n_r\nu^*(Q_{1,r})\mbox{ by Lemma \ref{Lemma7}.}\\
&>&0\mbox{ by (\ref{eq8})}.
\end{array}
\end{equation}
We have that  ${\rm Tr}_{K^s/K}(y_1\cdots y_rf)\in R$ and
$$
\nu^*({\rm Tr}_{K^s/K}(y_1\cdots y_rf)-f)>\nu^*(f)
$$
by equations (\ref{eq9}) and (\ref{eq10}).

We thus have established the following proposition.

\begin{Proposition}\label{PropSplit} Suppose that $K^*$ is a Galois extension of a field $K$, $\nu$ is a valuation of $K$ and $\nu^*$ is an extension of $\nu$ to $K^*$. Suppose that $R$ is a normal excellent  local ring with quotient field $K$ which is dominated by $\nu$. Then there exists a birational extension $R'$ of $R$ such that $R'$ is a normal local ring which is dominated by $\nu$ and if $R''$ is a normal local ring with quotient field $K$ which birationally dominates $R'$ and is dominated by $\nu$ and  if $S$ is the local ring of the normalization of $R''$ in the fixed field $K^s$ of the  decomposition group $G^s(V_{\nu^*}/V_{\nu})$, then 
$$
{\rm gr}_{\nu^*}(S)={\rm gr}_{\nu}(R).
$$
\end{Proposition}

Let $R^*=T_m$ be the local ring of the integral closure $T$ of $R$ in $K^*$ which is dominated by $\nu^*$. If $A$ is  a local ring, we will write $k(A)=A/m_A$ to denote the residue field of $A$. 

If $F$ is a finite field extension of a field $E$, then $F_s$ will denote the separable closure of $E$ in $F$, and $[F:E]_s$ will denote the degree $[F_s:E]$. Recall that $G^s(R^*/R)=G^s(V_{\nu^*}/V_{\nu})$ by Lemma \ref{Lemma8} and our assumptions on $R$.

\begin{Lemma}\label{Lemma9}  There exists a normal local ring $R'$ which birationally dominates $R$, such that $\nu$ dominates $R'$, and if $R''$ is a normal local ring which birationally dominates $R'$ and is dominated by $\nu$, then
\begin{equation}\label{eq14}
G^i((R'')^*/R'')=G^i(V_{\nu^*}/V_{\nu}),
\end{equation}
where $(R'')^*$ is the local ring of the integral closure of $R''$ in $K^*$ which is dominated by $\nu^*$.

Let $k((R'')^*)_s$ be the separable closure of $k(R'')$ in $k((R'')^*)$. 
We further have that any basis of $k((R'')^*)_s$ over $k(R'')$ is a basis of $k(V_{\nu^*})_s$ over $k(V_{\nu})$; in particular, $k((R'')^*)_s$ and $k(V_{\nu})$ are linearly disjoint over $k(R'')$.
\end{Lemma}

\begin{proof} The residue field $k(V_{\nu^*})$ is finite over $k(V_{\nu})$ by Corollary 2 on page 26 of \cite{ZS2}. Let $k(V_{\nu^*})_s$ be the separable closure of $k(V_{\nu})$ in $k(V_{\nu^*})$. Let $\overline t_1,\ldots,\overline t_h\in k(V_{\nu^*})_s$ be a $k(V_{\nu})$-basis. For $1\le i\le h$, let
$$
f_i(x)=x^{m_i}+\overline a_{1,i}x^{m_i-1}+\cdots+\overline a_{m_i,i}\in k(V_\nu)[x]
$$
be the minimal polynomial of $\overline t_i$ over $k(V_{\nu})$. Each $f_i(x)$ is a separable polynomial. 

Let $A=\cap_jV_{\nu_j^*}$ be the integral closure of $V_{\nu}$ in $K^*$. Let $t_1,\ldots,t_h$ be lifts of the $\overline t_i$ to $A$. Let 
$$
g_i(x)=x^{n_i}+b_{1,i}x^{n_i-1}+\cdots+b_{n_i,i}\in V_{\nu}[x]
$$
be such that $g_i(t_i)=0$ for $1\le  i\le h$. Let $a_{j,i}$ be lifts of the $\overline a_{j,i}$ to $V_{\nu}$ and let $B$ be the integral closure of 
$$
R[\{a_{j,i}\mid  1\le i\le h, 1\le j\le m_i\},\{b_{j,i}\mid 1\le i\le h, 1\le j\le n_i\}]
$$
in $K$ and let $R'=B_{B\cap m_{\nu}}$. Suppose that $R''$ is a normal local ring which birationally dominates $R'$ and is dominated by $\nu$. Let $(R'')^*$ be the local ring of the integral closure of $R''$ in $K^*$ which is dominated by $\nu^*$.

Now $t_i\in (R'')^*$ for $1\le i\le h$ and $f_i(x)\in k(R'')[x]$ for $1\le  i\le h$ since $R''$ dominates $R'$. Thus $\overline t_i\in k((R'')^*)$ is separable over $k(R'')$ for $1\le i\le h$ since $f_i(x)$ is a separable polynomial, and so
\begin{equation}\label{eq13}
h\le [k((R'')^*):k(R'')]_s.
\end{equation}
We have that $G^s((R'')^*/R'')=G^s(V_{\nu^*}/V_{\nu})$ (by Lemma \ref{Lemma8}), and $G^i(V_{\nu^*}/V_{\nu})\subset G^i((R'')^*/R'')$ by Proposition 1.50 \cite{RTM}.
By Theorem 1.48 \cite{RTM}, we have that $k((R'')^*)$ is a normal extension of $k(R'')$ and $k(V_{\nu^*})$ is a normal extension of $k(V_{\nu})$, with automorphism groups
$$
{\rm Aut}(k((R'')^*)/k(R''))\cong G^s((R'')^*/R'')/G^i((R'')^*/R'')
$$
and
$$
{\rm Aut}(k(V_{\nu^*})/k(V_{\nu}))\cong G^s(V_{\nu^*}/V_{\nu})/G^i(V_{\nu^*}/V_{\nu}).
$$
We then have a natural short exact sequence of groups
$$
0\rightarrow G^i((R'')^*/R'')/G^i(V_{\nu^*}/V_{\nu})\rightarrow {\rm Aut}(k(V_{\nu^*})/k(V_{\nu}))\rightarrow
{\rm Aut}(k((R'')^*)/k(R''))\rightarrow 0.
$$
Now $|{\rm Aut}(k(V_{\nu^*})/k(V_{\nu})|=h$ and 
$$
h\le [k((R'')^*:k(R'')]_s=|{\rm Aut}(k((R'')^*)/k(R''))|
$$
by (\ref{eq13}), so $G^i((R'')^*/R'')=G^i(V_{\nu^*}/V_{\nu})$. We further have that the basis $\overline t_1,\ldots,\overline t_h$ of $k(V_{\nu^*})_s$ over $k(V_{\nu})$ is a basis of $k((R'')^*)_s$ over $k(R'')$. Thus any basis of $k((R'')^*)_s$ over $k(R'')$ is a basis of $k(V_{\nu^*})_s$ over $k(V_{\nu})$.
\end{proof}

\begin{Proposition}\label{PropInert} Suppose that $K^*$ is a Galois extension of a field $K$, $\nu$ is a valuation of $K$ and $\nu^*$ is an extension of $\nu$ to $K^*$. Suppose that $R$ is a normal local ring with quotient field $K$ which is dominated by $\nu$. Then there exists a birational extension $R'$ of $R$ such that $R'$ is a normal local ring which is dominated by $\nu$ and if $R''$ is a normal local ring with quotient field $K$ which birationally dominates $R'$ and is dominated by $\nu$, and  if $R^s$ is the local ring of the normalization of $R''$ in the fixed field $K^s$ of the  decomposition group $G^s(V_{\nu^*}/V_{\nu})$ which is dominated by $\nu^*$ and  $R^i$ is the local ring of the normalization of $R''$ in the fixed field $K^i$ of the  inertia group $G^i(V_{\nu^*}/V_{\nu})$ which is dominated by $\nu^*$, then  we have a natural isomorphism of graded rings
$$
{\rm gr}_{\nu^*}(R^i)\cong{\rm gr}_{\nu^*}(R^s)\otimes_{k(R^s)}k(R^i).
$$
\end{Proposition}

\begin{proof} Let $R'$ be  a birational extension of $R$ which is a normal local ring and is dominated by $\nu$ and satisfies the conclusions of Lemmas \ref{Lemma10}, \ref{Crit} (which implies Lemma \ref{Lemma8}) and \ref{Lemma9}. Let $R''$ be a normal local ring which is a birational extension of $R'$ and is dominated by $\nu$. Without loss of generality, we may assume that $R$ satisfies the conclusions of Lemmas \ref{Lemma10}, \ref{Crit} and \ref{Lemma9} and $R=R''$. In particular, we have 
\begin{equation}\label{eq20}
G^s(R^*/R)\cong G^s(V_{\nu^*}/V_{\nu})\mbox{ and }G^i(R^*/R)\cong G^i(V_{\nu^*}/V_{\nu}).
\end{equation}
Let $\nu^i$ be the restriction of $\nu^*$ to the inertia field $K^i=K^{G^i(V_{\nu^*}/V_{\nu})}$, and let $\nu^s$ be the restriction of $\nu^*$ to the splitting field $K^s=K^{G^s(V_{\nu^*}/V_{\nu})}$. By Theorem 1.48 \cite{RTM} and (\ref{eq20}), we have that $R^s\rightarrow R^i$ is unramified, $k(R^i)=k(R^*)_s$ and $k(V_{\nu^i})=k(V_{\nu^*})_s$. By Theorem 1.47 \cite{RTM} and (\ref{eq20}), $k(V_{\nu^s})=k(V_{\nu})$ and $k(R^s)=k(R)$. 
Let $\overline a_1,\ldots,\overline a_n$ be a basis of $k(R^i)$ over $k(R^s)$. Then $\overline a_1,\ldots,\overline a_n$ is a basis of $k(V_{\nu^i})$ over $k(V_{\nu^s})$ by Lemma \ref{Lemma9}. Let $a_1,\ldots,a_n\in R^i$ be lifts of $\overline a_1,\ldots,\overline a_n$. The  extension $R^s\rightarrow R^i$ is unramified,  so $m_{R^s}R^i=m_{R^i}$. Thus 
$$
a_1R^s+\cdots+a_nR^s+m_{R^s}R^i=R^i.
$$
Now $R^i$ is the unique local ring of $K^i$ lying over $R^s$ by Proposition 1.46 \cite{RTM}, so $R^i$ is the integral closure of the excellent local ring $R^s$ in $K^i$, and so $R^i$ is a finitely generated $R^s$-module.
Thus $R^i=a_1R^s+\cdots+a_nR^s$ by Nakayama's lemma. Let $0\ne f\in R^i$. Then $f=a_1g_1+\cdot+a_ng_n$ with $g_1,\ldots,g_n\in R^s$. Let $\lambda=\min\{\nu^s(g_i)\}$. After possibly reindexing the $g_i$, we may assume that $\nu(g_i)=\lambda$ for $1\le i\le r$ and $\nu(g_i)>\lambda$ if $r<i\le n$. Let $b_i$ be the residue of $\frac{g_i}{g_1}$ in $k(V_{\nu^s})$ for $1\le i\le r$, which are necessarily all nonzero, so we have that
$\sum_{i=1}^r\overline a_ib_i\ne 0$ in $k(V_{\nu^i})$, since $\overline a_1,\ldots,\overline a_r$ are linearly independent over $k(V_{\nu})=k(V_{\nu^s})$. Thus
$$
\nu^*\left(\sum_{i=1}^ra_i\frac{g_i}{g_1}\right)=0,
$$
so that 
$$
\nu^*(f)=\nu^*(\sum_{i=1}^ra_ig_i)=\nu^*(g_1)=\cdots=\nu^*(g_r)
$$
and
$$
\nu^*(f-a_1g_1-\cdots -a_rg_r)>\nu^*(f).
$$
Thus we have equality of semigroups $S^{R^s}(\nu^s)=S^{R^i}(\nu^i)$ and for all $\gamma\in S^{R^s}(\nu^s)$ we have a natural isomorphism
$$
\mathcal P_{\gamma}(R^i)/\mathcal P_{\gamma}^+(R^i)\cong [\mathcal P_{\gamma}(R^s)/\mathcal P_{\gamma}^+(R^s)]\otimes_{k(R^s)}k(R^i),
$$
showing that
$$
{\rm gr}_{\nu^i}(R^i)\cong {\rm gr}_{\nu^s}(R^s)\otimes_{k(R^s)}k(R^i).
$$
\end{proof}

\section{Some basic results on ramification}
In this section we  extract some results from Abhyankar's paper \cite{LU}.
 Suppose that $K$ is a field  and $K'$ is a finite separable extension of $K$.  Suppose that $R$ is a normal, excellent  local ring with quotient field   $K$ and $R'$ is a normal local ring of $K'$ which lies over $R$. Let $E$ be the quotient field of $\hat R$ and $E'$ be the quotient field of $\hat R'$. Define
 $$
 d(R':R)=[E':E],\,\, g(R':R)=[R'/m_{R'}:R/m_R]_s\mbox{ and }r(R':R)=d(R':R)/g(R':R).
 $$
 We have that $d(R':R), g(R':R), r(R':R)$ are all multiplicative in towers of fields.

 Now suppose that $K^*$ is a finite Galois extension of $K$ and that $R^*$ is a normal local ring of $K^*$ which lies over $R$.  Let $R^s=R^*\cap K^s$ where $K^s=(K^*)^{G^s(R^*/R)}$ is the splitting field of $R^*$ over $R$.  Let $R^i=R^*\cap K^i$ where $K^i$ is  the inertia field $K^i=(K^*)^{G^i(R^*/R)}$ of $R^*$ over $R$. Let $E, E^s, E^i, E^*$ be the respective quotient fields of the complete local rings $\hat R, \widehat{R^s}, \widehat{R^i}$ and $\widehat{R^*}$.

 \begin{Proposition}\label{PropLG1}  Let $R$ be a normal excellent local  ring. Let $K$ be the quotient field of $R$. Let  $K^*$ be a finite separable extension of $K$ and let  $n=[K^*:K]$. Let $x\in R$ be a primitive element of $K^*$ over $K$, with minimal polynomial $f(t)\in R[t]$ (such an $x$ exists by Theorem 4, page 260 \cite{ZS1}). 
 Let $S_i$ for $1\le i\le h$ be the local rings in $K^*$ lying over $R$ and $S$ be the integral closure of $R$ in $K^*$. Then $\hat R$ and $\hat S_i$ are  normal local domains, the natural homomorphisms $\hat R\rightarrow \hat S_i$
 are injective for all $i$ and we have a natural isomorphism
 $$
 S\otimes_R\hat R\cong \oplus_{i=1}^h \hat S_i.
 $$
 Let $E$ be the quotient field of $\hat R$ and $E_i$ be the quotient field of $\hat S_i$ for $1\le i\le h$. Let $e_i=[E_i:E]$. Then $n=e_1+\cdots+e_h$. 
 Further, $x$ is a primitive element of $E_i$ over $E$ for all $i$ with minimal polynomial $f_i(t)\in E[t]$  and  there is a factorization 
 $f(t)=f_1(t)\cdots f_h(t)$ in $E[t]$.
 \end{Proposition} 
 
 \begin{proof} We have that $\hat R$ is a normal local domain since $R$ is normal and excellent. By Theorem 16, page 277 \cite{ZS2} and Corollary 2, page 283 \cite{ZS2}, we have a natural isomorphism
 $$
  S\otimes_R\hat R\cong \oplus \hat S_i.
 $$
  We have that the $\hat S_i$ are normal local domains since the $S_i$ are normal and excellent.
    We have that 
  $$
  K^*\otimes_KE\cong E[x]\cong E[t]/(f(t)).
  $$
  
  Let $A$ be a ring. The total quotient ring of $A$ is ${\rm QR}(A)=S^{-1}A$ where $S$ is the multiplicative set of all non zero divisors  of $A$. Let $g\in K^*\otimes_KE=E[x]$. Then $g=\frac{a}{b}$ where $a\in \hat R[x]$ and $b\in \hat R\setminus \{0\}$. Since $\hat R$ is a domain, $b$ is not a zero divisor in $S\otimes_R\hat R$ by Theorem 16, page 277 \cite{ZS2}. 
 Thus $K^*\otimes_KE\subset {\rm QR}(S\otimes_R\hat R)$.  Since the reduced ring $S\otimes_R\hat R$ is naturally a subring of ${\rm QR}(S\otimes_R\hat R)$, we have a natural inclusion $S\otimes_R\hat R\subset K^*\otimes_KE$.
  Now $f(t)$ is reduced in $E[t]$ since $K^*$ is separable over $K$. We have that $K^*\otimes_KE\cong \oplus E[t]/(f_i(t))$ is a direct sum of fields, where $f_i(t)$ are the irreducible factors of $f(t)$ in $E[t]$ by the Chinese remainder theorem.  Thus ${\rm QR}(S\otimes_R\hat R)=K^*\otimes_KE$. Now $S\otimes_R\hat R$ is reduced, so ${\rm QR}(S\otimes_R\hat R)\cong \oplus_iE_i$. Thus after reindexing, we have that $E_i\cong E[t]/(f_i(t))$ and we have that
  $$
  \sum [E_i:E]=\sum \deg (f_i)=\deg (f) =[K^*:K].
  $$
 \end{proof}

 Since $R^*$ is the unique local ring of $K^*$ which dominates $R^s$ by Proposition 1.46 \cite{RTM}, we have by Proposition \ref{PropLG1} that 
  \begin{equation}\label{eqLG5}
 [K^*:K^s]=d(R^*:R^s)\mbox{ and } [K^*:K^i]=d(R^*:R^i).
 \end{equation}
 By Proposition \ref{PropLG1} and Theorem 1.48 \cite{RTM},  we have that
 \begin{equation}\label{eqLG6}
 [K^i:K^s]=d(R^i:R^s)=g(R^*:R).
 \end{equation}
 We have that $\hat R^s\cong\hat R$ by  Theorem 1.47 \cite{RTM}  and Theorem 30.6 \cite{N}, and so       $E^s\cong E$. Thus
 \begin{equation}\label{eqLG7}
 d(R^s:R)=1.
 \end{equation}
 
 \begin{Lemma}\label{LemmaLG8} Let notations be as above. Then $r(R':R)$ is a positive integer. 
 \end{Lemma}
 
 \begin{proof} Let $K^*$ be a finite Galois extension of $K$ which contains $K'$, and let $R^*$ be a normal local ring of $K^*$ such that $R^*$ lies over $R'$. 
 We have that 
 $$
 G^s(R^*/R')=G^s(R^*/R)\cap G(K^*/K')\mbox{ and }G^i(R^*/R')=G^i(R^*/R)\cap G(K^*/K'),
 $$
 so $(K')^s=K'K^s$ and $(K')^i=K'K^i$.
 Thus we have a commutative diagram of fields
 $$
 \begin{array}{ccccc}
 &&K^*&&\\
 &\nearrow&&\nwarrow&\\
 K^i&&\rightarrow &&(K')^i\\
 \uparrow&&&&\uparrow\\
 K^s&&\rightarrow&&(K')^s\\
 \uparrow&&&&\uparrow\\
 K&&\rightarrow&& K'
 \end{array}
 $$
 where $K^s=(K^*)^{G^s(R^*/R)}$, $K^i=(K^*)^{G^i(R^*/R)}$, $(K')^s=(K^*)^{G^s(R^*/R')}$ and $(K')^i=(K^*)^{G^i(R^*/R')}$.
 Considering the induced commutative diagram of fields  
 $$
 \begin{array}{ccccc}
 &&E^*&&\\
 &\nearrow&&\nwarrow&\\
 E^i&&\rightarrow &&(E')^i\\
 \uparrow&&&&\uparrow\\
 E^s&&\rightarrow&&(E')^s\\
 \uparrow&&&&\uparrow\\
 E&&\rightarrow&& E'
 \end{array}
 $$ 
 where $E, E', E^*, E^s, E^i, (E')^s, (E')^i$ are the respective quotient fields of $\hat R, \widehat{R'}, \widehat{R^*}, \widehat{R^s}, \widehat{R^i}, \widehat{(R')^s}, \widehat{(R')^i}$, we see from (\ref{eqLG6}) and (\ref{eqLG7}) that
 $$
 d(R':R)g(R^*:R')=g(R^*:R)d((R')^i:R^i).
 $$
 Thus $d(R':R)=g(R':R)d((R')^i:R^i)$.
 \end{proof}
  
 We also read off the following formulas from the commutative diagrams of Lemma \ref{LemmaLG8} and from (\ref{eqLG5}), (\ref{eqLG6}) and (\ref{eqLG7}):
 $$
 [K^*:K^i][R^*/m_{R^*}:R/m_R]_s = d(R':R)[K^*:(K')^i][R^*/m_{R^*}:R'/m_{R'}]_s, 
 $$
 so that
 \begin{equation}\label{eqLG9}
 [K^*:K^i]=r(R':R)[K^*:(K')^i]
 \end{equation}
 and
 \begin{equation}\label{eqLG10}
 [K^*:K^s]=d(R':R)[K^*:(K')^s].
 \end{equation}
 
 \begin{Lemma}\label{LemmaGL11} Let notation be as above. Then 
 \begin{enumerate}
 \item[1)] $G^i(R^*/R)\subset G(K^*/K')$ if and only if $r(R':R)=1$ and
 \item[2)] $G^s(R^*/R)\subset G(K^*/K')$ if and only if $d(R':R)=1$. 
 \end{enumerate}
 \end{Lemma}
 
 \begin{proof} The lemma follows from equations (\ref{eqLG9}) and (\ref{eqLG10}) and the observations that $G^i(R^*/R)\subset G(K^*/K')$ if and only if $(K')^i=K^i$ and  $G^s(R^*/R)\subset G(K^*/K')$ if and only if
  $(K')^s=K^s$. 
  \end{proof}

 \begin{Proposition}\label{PropGL12} Suppose that $R$ is an excellent normal local ring with quotient field $K$, $K'$ is a finite field extension of $K$ and $R'$ is a normal local ring of $K'$ which lies over $R$.  Then $R\rightarrow R'$ is unramified if and only if $r(R'/R)=1$.
 \end{Proposition}
 
 \begin{proof} The extension $R\rightarrow R'$ is unramified if and only if $\hat R\rightarrow \hat R'$ is unramified which holds if and only if $D(\hat R'/\hat R)=\hat R$ by Theorems 1.44 and  1.44A \cite{RTM} (the proof of Theorem 1.44A, which is given in \cite{LU}, holds when $R$ is excellent). The discriminant ideal $D(\hat R'/\hat R)$ is defined on page 31 of \cite{RTM}. This condition holds if and only if 
 $$
 [R'/m_{R'}:R/m_R]_s=[E':E],
 $$
  where $E'$ is the quotient field of $\hat R'$ and $E$ is the quotient field of $\hat R$ by Theorem 1.45 \cite{RTM}. But this is equivalent to the condition that $r(R'/R)=1$.
 \end{proof}

\section{The extension of associated graded rings in an unramified extension}\label{SecUnR}

In this section we prove Proposition \ref{PropUnR}. We have the following assumptions. Suppose that $R$ and $S$ are  normal local rings such that $R$ is excellent, $S$ lies over  $R$, $\tilde \nu$ is a valuation of the quotient field $L$ of $S$ which dominates $S$, and $\nu$ is the restriction of $\tilde \nu$ to the quotient field $K$ of $R$. Suppose that $L$ is finite separable over $K$.  

\begin{Lemma}\label{Lemma4}  Suppose that $S_0$ is a local ring which is a birational extension of $S$ and is dominated by $\tilde\nu$. Then there exists a normal local ring $R'$ which is a birational extension of $R$ and is dominated by $\nu$, which has the property that if $R''$ is a normal local ring which is a birational extension of $R'$ and is dominated by $\nu$, and if $S''$ is the normal local ring of $L$ which lies over $R''$ and is dominated by $\tilde\nu$, then $S''$ dominates $S_0$.
\end{Lemma}

\begin{proof} There exist $f_1,\ldots,f_n\in V_{\tilde \nu}$ such that $S_0$ is a localization of $S[f_1,\ldots,f_n]$. Let $A$ be the integral closure of $V_{\nu}$ in $L$ so that $V_{\tilde\nu}$ is the localization of $A$ at $m_{\tilde\nu}\cap A$. Thus we may write, for $1\le i\le n$, $f_i=\frac{b_i}{c_i}$ with $b_i,c_i\in A$, $\tilde\nu(b_i)\ge 0$ and $\tilde \nu(c_i)=0$.
For $1\le i\le n$, let
$$
b_i^{m_i}+d_{i,1}b_i^{m_i-1}+\cdots+ d_{i,m_i}=0\mbox{ and }
c_i^{n_i}+e_{i,1}c_i^{n_i-1}+\cdots+e_{i,n_i}=0
$$
be equations of integral dependence of $b_i$ and $c_i$ over $V_{\nu}$, so that all $d_{i,j},e_{i.j}\in V_{\nu}$.
Let $B$ be the integral closure of 
$$
R[\{d_{i,j}\mid 1\le j\le m_i, 1\le i\le n\},\{e_{i,j}\mid 1\le j\le n_i, 1\le i\le n\}]
$$ 
in $K$, and let $R'=B_{m_{\nu}\cap B}$. Suppose that $R''$ and $S''$ are as in the statement of the lemma. Then $S''$ is the localization of the integral closure $C$ of $R''$ in $L$ at $C\cap m_{\tilde\nu}$. We have that $b_i,c_i\in C$ for $1\le i\le n$ and $c_i\not\in m_{\tilde\nu}\cap C$ for $1\le i\le n$ so $f_1,\ldots,f_n\in S''$. Thus $S_0\subset S''$ and $S''$ dominates $S_0$.
\end{proof}

We now impose the further condition that $S$ is unramified over $R$, and   prove Proposition \ref{PropUnR}.

 Let $K^*$ be a Galois closure of $L$ over $K$. Let $\nu^*$ be an extension of $\tilde\nu$ to $K^*$. Let $R^*$ be the local ring of $K^*$ which is dominated by $\nu^*$ and lies over $S$. Now $G^i(R^*/R)\subset G(K^*/L)$ by Lemma \ref{LemmaGL11} and Proposition \ref{PropGL12}, since $R\rightarrow S$ is unramified. Further, $G^i(V_{\nu^*}/V_{\nu})\subset G^i(R^*/R)$ by Proposition 1.50 \cite{RTM}.
Thus $G^i(V_{\nu^*}/V_{\nu})\subset G(K^*/L)$, so $L\subset K^i=(K^*)^{G^i(V_{\nu^*}/V_{\nu})}$.

By Lemmas \ref{Lemma10}, \ref{Crit}  (which implies Lemma \ref{Lemma8}) and \ref{Lemma9}, there exists a normal local ring $R'$ which is a birational extension of $R$ and is dominated by $\nu$ such that the conclusions of Lemmas \ref{Lemma10}, \ref{Crit} and \ref{Lemma9} hold for the field extension $K\rightarrow K^*$, $V_{\nu^*}$ over $V_{\nu}$ and $R'$. Further, 
by Lemmas \ref{Lemma10}, \ref{Crit} and \ref{Lemma9}, there exists a normal local ring $S'$ which is a birational extension of $S$ and is dominated by $\tilde\nu$ such that the conclusions of Lemmas \ref{Lemma10}, \ref{Crit} and \ref{Lemma9} hold for the field extension $L\rightarrow K^*$, $V_{\nu^*}$ over $V_{\tilde \nu}$ and $S'$. 
By Lemma \ref{Lemma4}, we may assume that the normal local ring of $L$ which lies over $R'$ and is dominated by $\tilde\nu$ dominates $S'$. 
Replacing $R$ with $R'$ and $S$ with the normal local ring of $L$ which lies over $R'$ and is dominated by $\tilde\nu$, we may assume that the conclusions of Lemmas \ref{Lemma10}, \ref{Crit} and \ref{Lemma9} hold for $R$ and $S$.

 Let $L^s=(K^*)^{G^s(V_{\nu^*}/V_{\tilde\nu})}$, $K^s=(K^*)^{G^s(V_{\nu^*}/V_{\nu})}$ and $K^i=(K^*)^{G^i(V_{\nu^*}/V_{\nu})}$. We have that 
 $$
 G(K^*/K^sL)=G(K^*/L)\cap G^s(V_{\nu^*}/V_{\nu})=G^s(V_{\nu^*}/V_{\tilde \nu})
 $$
 so $L^s$ is the join $L^s=K^sL$.
 We then have a commutative diagram of fields
$$
\begin{array}{lllllllll}
&&L\\
&\nearrow&&\searrow\\
K&&&&L^s=K^sL&\rightarrow &K^i&\rightarrow&K^*.\\
&\searrow&&\nearrow\\
&&K^s\\
\end{array}
$$
Let $S^s=L^s\cap R^*$,  $R^s=K^s\cap R^*$ and $R^i=K^i\cap R^*$. Then
$$
{\rm gr}_{\tilde\nu}(S)\cong {\rm gr}_{\nu^*}(S^s)
$$
by Proposition \ref{PropSplit}, and $S^s/m_{S^s}\cong S/m_S$. Also,
$$
{\rm gr}_{\nu^*}(R^s)\cong {\rm gr}_{\nu}(R)
$$
 by Proposition \ref{PropSplit}, with $R^s/m_{R^s}\cong R/m_R$. 
 We have that $S^s$ is finite over $R^s$ since $S^s$ is the unique local ring of $L^s$ lying over $R^s$, and the elements of a basis of $S/m_S$ over $R/m_R$ are linearly independent over $k(V_{\nu})$ since $k(V_{\nu^*})$ and $k(R^i)$ are linearly disjoint over $k(R)$ by Lemma \ref{Lemma9}.
 Thus the proof of Proposition \ref{PropInert} shows that
$$
{\rm gr}_{\nu^*}(S^s)\cong {\rm gr}_{\nu^*}(R^s)\otimes_{R/m_R}S/m_S.
$$

We now give an example showing that taking a birational extension of $R$ may be necessary to obtain the conclusions of Proposition \ref{PropUnR}. 
Let $k$ be a field of characteristic $\ne 2$ and $k(x,y,z)$ be a three dimensional rational function field. Let $k'=k(z)$, which is a one dimensional   rational function field. Let $\nu_1$ be the $k'$-valuation of $k'(x,y)$ which  is determined by a generating sequence $P_0=x, P_1=y, P_2=y^2-(z+1)x^2,P_3,\ldots$ in $k'[x,y]$ where $\nu_1(P_0)=\nu_1(P_1)=1$, $\nu_1(P_2)=\frac{5}{2}$ and all the $P_i$ are in $k[x,y,z]$ (for instance using the algorithm of \cite{CV1}).We have that $\nu_1(P_{i+1})>\nu_1(P_i)$ if $i>1$  and the semigroup $S^{k'[x,y]_{(x,y)}}(\nu_1)$ is generated by $\nu(P_0), \nu(P_1), \nu(P_2),\nu(P_3),\ldots$. Let $\nu_2$ be an extension to $V_{\nu_1}/m_{\nu_1}$ of the $k$-valuation of $k'$ defined by $\nu_2(g(z))=\mbox{ord}_z(g)$ for $g\in k[z]$. Let $\nu$ be the composition of $\nu_1$ and $\nu_2$, and let $R=k[x,y,z]_{(x,y,z)}$ which is dominated by $\nu$. We have that $R_{(x,y)}=k'[x,y]_{(x,y)}$.  

Let $S=(R[w]/w^2-(z+1))_{(x,y,z,w-1)}$.
 The extension $R\rightarrow S$ is unramified. Let $\nu^*$ be an extension of $\nu$ to $K^*={\rm QF}(S)$ which dominates $S$ and is composite with extensions $\overline\nu_1$ of $\nu_1$ to $K^*$ and $\overline\nu_2$ of $\nu_2$ to $V_{\overline\nu_1}/m_{\overline\nu_1}$. 
 
 By the explanation on page 56 \cite{RTM} or Theorem 17, page 43 \cite{ZS2}, we have a commutative diagram of homomorphisms of value groups, where the horizontal sequences are short exact and the vertical arrows are injective,
 $$
 \begin{array}{ccccccccc}
 0&\rightarrow &\Phi_{\nu_2}&\rightarrow&\Phi_{\nu}&\rightarrow&\Phi_{\nu_1}&\rightarrow& 0\\
 &&\downarrow&&\downarrow&&\downarrow\\
 0&\rightarrow &\Phi_{\overline \nu_2}&\rightarrow&\Phi_{\nu^*}&\rightarrow&\Phi_{\overline \nu_1}&\rightarrow& 0
 \end{array}
 $$
  which induces a commutative diagram of homomorphisms of semigroups, where the horizontal arrows are surjective and the vertical arrows are injective,
  $$
  \begin{array}{ccc}
  S^R(\nu)&\rightarrow&S^{R_{(x,y)}}(\nu_1)\\
  \downarrow&&\downarrow\\
  S^S(\nu^*)&\rightarrow&S^{S_{(x,y)}}(\overline\nu_1).
  \end{array}
  $$
 
We have that 
 $$
 \frac{5}{2}=\nu_1(P_2)=\overline \nu_1(y-wx)+\overline \nu_1(y+wx).
  $$
  Letting $\overline{\left(\frac{y}{x}\right)}$ and $\overline w$ be the respective classes of $\frac{y}{x}$ and $w$ in $V_{\overline \nu_1}/m_{\overline\nu_1}$, we deduce from $P_2$ that $\overline{\left(\frac{y}{x}\right)}=\pm \overline w$. Without loss of generality, $\overline{\left(\frac{y}{x}\right)}=\overline w$. Then   $\overline\nu_1(y+wx)=1$. Thus
 $$
 \overline\nu_1(y-wx)=\frac{3}{2}\not\in S^{R_{(x,y)}}(\nu_1).
 $$

\section{Conventions on valuations in algebraic function fields}\label{convs}
We recall some classical invariants of  valuations (Chapter VI, \cite{ZS2},
\cite{RTM}), and establish some notation which we will follow for the remainder of the paper.

Suppose that $K$ is a field of algebraic functions over a field $k$.
We will say that
a local domain  $R$ with quotient field  $K$ is an algebraic local ring of $K$ if $R$ is a localization of a  finite type  $k$-algebra.
A valuation $\nu$ of $K$ will be called a $k$-valuation if $\nu(f)=0$ for all nonzero $f\in k$. 
If $X$ is an integral
$k$-scheme with function field $K$, then a point $p\in X$ is called a center of the
valuation $\nu$ (or the valuation ring $V_{\nu}$) if $V_{\nu}$ dominates ${\mathcal O}_{X,p}$.

Suppose that 
  $\nu$ is a $k$-valuation of $K$ with valuation ring $V_{\nu}$ and value group $\Phi_{\nu}$. Let $m_{\nu}$ be the maximal ideal of $V_{\nu}$.
The rank $r$ of $V_{\nu}$ is finite since 
$r\le \text{trdeg}_kK<\infty$ by the Corollary on  page 50 of  Section 11, Chapter VI, \cite{ZS2}. 
We have the chains of isolated subgroups $\Phi_i$ in $\Phi_{\nu}$ of (\ref{eqVT1}) and of prime ideal $I_i$ in $V_{\nu}$ of (\ref{eqVT2}).

For $i<j$,
$(V_{\nu}/I_i)_{I_j}$ is a rank $j-i$ valuation ring with value group $\Phi_i/\Phi_j$
and with quotient field $(V_{\nu}/I_i)_{I_i}$.
$V_{\nu}$ is  said to be composite with the valuation ring  $(V_{\nu}/I_i)_{I_j}$. Set
$$
\lambda_i=\text{trdeg}_k(V_{\nu}/I_i)_{I_i}
$$
for $0\le i\le r$. The rational rank of $(V_{\nu}/I_{i-1})_{I_i}$ is
$$
s_i=\text{ratrank}(V_{\nu}/I_{i-1})_{I_i}:=\text{dim}_{\bold Q}(\Phi_{i-1}/\Phi_i)\otimes
{\bold Q}
$$
for $1\le i\le r$. The numbers $s_i$ and $\lambda_i$ are $<\infty$ by Theorem 1 \cite{Ab1} or by
Proposition 2, Appendix 2 \cite{ZS2}.

Now suppose that $K^*$ is a finite extension of $K$, and $\nu^*$ is 
an extension of $\nu$ to $K^*$.  Let $V_{\nu^*}$ be the valuation ring of $\nu^*$, and let $\Phi_{\nu^*}$ be the value group.
Recall from Section \ref{SecVT} that the prime ideals  of $V_{\nu^*}$ are
a finite chain
$$
0=I_0^*\subset\cdots\subset I_r^*=m_{\nu^*}\subset V_{\nu^*}
$$
with $I_i^*\cap V_{\nu}=I_i$, $0\le i\le r$, and with isolated subgroups
$$
0=\Phi_r^*\subset \cdots\subset \Phi_0^*=\Phi_{\nu^*}
$$
which have the property that $\Phi_i^*\cap\Phi_{\nu}=\Phi_i$
for $0\le i\le r$ and $\Phi_i^*/\Phi_i$ is a finite (Abelian) group for 
$0\le i\le r$ (Section 11, Chapter VI \cite{ZS2}). We further have that 
$$
\text{trdeg}_k(V_{\nu^*}/I_i^*)_{I_i^*}=\text{trdeg}_k(V_{\nu}/I_i)_{I_i}=\lambda_i
$$
for $0\le i\le r$ and
$$
\text{ratrank}(V_{\nu^*}/I^*_{i-1})_{I^*_i}=\text{ratrank}(V_{\nu}/I_{i-1})_{I_i}=s_i
$$
for $1\le i\le r$. Set $t_i=\lambda_{i-1}-\lambda_i$ for $1\le i\le r$.

The reduced ramification index of $\nu^*$ relative to $\nu$ is defined to be 
(page 53, Section 11, Chapter VI, \cite{ZS2}) 
\begin{equation}\label{eq31}
e=[\Phi_{\nu^*}:\Phi_{\nu}].
\end{equation}
The residue degree of $\nu^*$ with respect to $\nu$ is defined to be (page 53, Section 11, Chapter VI \cite{ZS2}) 
\begin{equation}\label{eq32}
f=[V_{\nu^*}/m_{\nu^*}:V_{\nu}/m_{\nu}].
\end{equation}

\section{An Abyankar Jung Theorem along a valuation}\label{SecAJ}

In this section, we prove Theorem \ref{ThmAJV} and Proposition \ref{PropAJV2}. Let notation be as in the statement of Theorem \ref{ThmAJV}.

Let $u_1,\ldots,u_{\lambda_r}\in V_{\nu}$ be such that their classes $\overline u_1,\ldots,\overline u_{\lambda_r}$ in $V_{\nu}/m_{\nu}$ are a transcendence basis of $V_{\nu}/m_{\nu}$ over $k$. Then $k(u_1,\ldots,u_{\lambda_r})$ is a rational function field over $k$ which is contained in $V_{\nu}$ and $V_{\nu}/m_{\nu}$ is an algebraic extension of $k(u_1,\ldots,u_{\lambda_r})$. Let $T''$ be the integral closure of $\tilde R[u_1,\ldots,u_{\lambda_r}]$ in $K$, and $R''=T''_{T''\cap m_{\nu}}$.
 Suppose that $f\in k[u_1,\ldots,u_{\lambda_r}]$ is nonzero. Then $\nu(f)=0$ since $u_1,\ldots,u_{\lambda_r}$ are algebraically independent over $k$ in $V_{\nu}/m_{\nu}$. Thus $f\not\in m_{R''}=m_{\nu}\cap R''$, and so $\frac{1}{f}\in R''$. In particular, $k(u_1,\ldots,u_{\lambda_r})\subset R''$. We may thus replace $k$ with $k(u_1,\ldots,u_{\lambda_r})$ and $\tilde R$ with $R''$ allowing us to assume that $V_{\nu}/m_{\nu}$ is algebraic over $k$. Observe that $\tilde R/m_{\tilde R}$ is then necessarily a finite field extension of $k$ since $\tilde R/m_{\tilde R}$ is a finitely generated algebraic field extension of $k$.

We will use the following generalization of the strong monomialization theorem, Theorem 4.8 \cite{CP}. Theorem 4.8 \cite{CP} is itself a generalization of the local monomialization theorem of \cite{C}.

\begin{Theorem}\label{SSM}
Let $k$ be a field of characteristic zero, $K$ an algebraic function field over $k$, $K^*$ a finite algebraic extension of $K$,
$\nu^*$ a $k$-valuation of $K^*$. Suppose that $\tilde S$ is an algebraic  local ring with
quotient field $K^*$ which is dominated by $\nu^*$ and $\tilde R$ is an algebraic local ring with
quotient field $K$ which is dominated by $\tilde S$. Let notation be as in  Section \ref{convs} for
$V=V_{\nu}$, $V^*=V_{\nu^*}$. Then there exists a commutative diagram 
\begin{equation}
\begin{array}{lllll}
R_0&\rightarrow& S&\subset&V^*\\
\uparrow&&\uparrow\\
\tilde R&\rightarrow&\tilde S
\end{array}
\end{equation}
such that  
 $\tilde R \rightarrow R_0$ and $\tilde S \rightarrow S$ are sequences of monoidal transforms
such that $V^*$ dominates $S$, $S$ dominates $R_0$ and 
there are regular parameters $(x_1, .... ,x_n)$
in $R_0$,  $(y_1, ... ,y_n)$ in $S$ such that 
$$
\begin{array}{l}
I_i\cap R_0=(x_1,\ldots,x_{t_1+\cdots+t_i})\\
I_i^*\cap S=(y_1,\ldots,y_{t_1+\cdots+t_i})
\end{array}
$$
for $1\le i\le r$ and there are relations
$$
\begin{array}{ll}
x_1&=y_1^{g_{11}(1)}\cdots y_{s_1}^{g_{1s_1}(1)}y_{t_1+1}^{h_{1,t_1+1}(1)}\cdots
y_n^{h_{1n}(1)}\delta_{11}\\
&\vdots\\
x_{s_1}&=y_1^{g_{s_11}(1)}\cdots y_{s_1}^{g_{s_1s_1}(1)}y_{t_1+1}^{h_{s_1,t_1+1}(1)}
\cdots y_n^{h_{s_1n}(1)}\delta_{1s_1}\\
x_{s_1+1}&=y_{s_1+1}\\
&\vdots\\
x_{t_1}&=y_{t_1}\\
x_{t_1+1}&=y_{t_1+1}^{g_{11}(2)}\cdots y_{t_1+s_2}^{g_{1s_2}(2)}y_{t_1+t_2+1}^{h_{1,t_1+t_2+1}(2)}
\cdots y_n^{h_{1n}(2)}\delta_{21}\\
&\vdots\\
x_{t_1+s_2}&=y_{t_1+1}^{g_{s_21}(2)}\cdots y_{t_1+s_2}^{g_{s_2s_2}(2)}
y_{t_1+t_2+1}^{h_{s_2,t_1+t_2+1}(2)}\cdots y_n^{h_{s_2n}(2)}\delta_{2s_2}\\
x_{t_1+s_2+1}&=y_{t_1+s_2+1}\\
&\vdots\\
x_{t_1+t_2}&=y_{t_1+t_2}\\
&\vdots\\
x_{t_1+\cdots +t_{r-1}+1}&=y_{t_1+\cdots+t_{r-1}+1}^{g_{11}(r)}\cdots y_{t_1+\cdots+t_{r-1}+s_r}
^{g_{1s_r}(r)}\delta_{r1}\\
&\vdots\\
x_{t_1+\cdots+t_{r-1}+s_r}&=y_{t_1+\cdots+t_{r-1}+1}^{g_{s_r1}(r)}\cdots
y_{t_1+\cdots+t_{r-1}+s_r}^{g_{s_rs_r}(r)}\delta_{rs_r}\\
x_{t_1+\cdots+t_{r-1}+s_r+1}&=y_{t_1+\cdots+t_{r-1}+s_r+1}\\
&\vdots\\
x_{t_1+\cdots+t_r}&=y_{t_1+\cdots+t_r}
\end{array}
$$
where $n=t_1+\cdots+t_r$ and for $1\le i\le r$,
$$
\text{det}\left(\begin{array}{lll}
g_{11}(i)&\cdots&g_{1s_i}(i)\\
\vdots&&\vdots\\
g_{s_i1}(i)&\cdots&g_{s_is_i}(i)
\end{array}\right)\ne 0,
$$
$\delta_{ij}$ are units in $S$, $h_{i,j}(l)$ are natural numbers such that for
$1\le l\le k\le r-1$ and $1\le i\le s_l$,
$$
h_{i,j}(l)=0\text{ if }t_1+\cdots+t_k+s_{k+1}<j\le t_1+\cdots+t_{k+1}.
$$
Let
$$
T=\{j\mid t_1+\cdots+t_k<j\le t_1+\cdots+t_k+s_{k+1}\text{ for some }0\le k\le r-1\}.
$$
Then $\{\nu^*(y_j)\mid j\in T\}$ is a rational basis of $\Phi_{\nu^*}\otimes{\bold Q}$ and
$\{\nu^*(x_j)\mid j\in T\}$ is a rational basis of $\Phi_{\nu}\otimes{\bold Q}$.
\end{Theorem}

The statement of Theorem \ref{SSM} differs from the statement of Theorem 4.8 \cite{CP} in that we have the stronger statement that
$$
x_{t_1+\cdots+t_{i-1}+j}=y_{t_1+\cdots+t_{i-1}+j}
$$
for $1\le i\le r$ (with the convention $t_1+\cdots+t_{i-1}=0$ if $i=0$) and for $s_i<j\le t_i$. We continue to have that
$\nu(x_{t_1+\cdots+t_{i-1}+1}),\ldots,\nu(x_{t_1+\cdots+t_{i-1}+s_i})$ is a basis of the $\QQ$-vector space $G(\nu(x_{t_1+\cdots+t_{i-1}+1}),\ldots,\nu(x_{t_1+\cdots+t_{i-1}+t_i}))\otimes_{\ZZ}\QQ$.

\begin{proof} We may assume that the conclusions of Theorem 4.8 \cite{CP} hold. For $i<k$, $1\le j\le t_i$, $1\le l\le t_k$, define a monoidal transform $S(i,j,k,l)$ from $S$ to $S'$ along $\nu^*$, where $S'$ has regular parameters $y_m'$ by
$$
y_m=\left\{\begin{array}{ll}
y_{t_1+\cdots+t_{i-1}+j}'y_{t_1+\cdots+t_{k-1}+l}&\mbox{ if }m=t_1+\cdots+t_{i-1}+j\\
y_m'&\mbox{ if }m\ne t_1+\cdots+t_{i-1}+j.
\end{array}\right.
$$
The monoidal transform $S(i,j,k,l)$ is along $\nu^*$ since $\nu^*(y_{t_1+\cdots+t_{i-1}+j})>\nu^*(y_{t_1+\cdots+t_{k-1}+l})$ as $i<k$.  

We will prove the theorem by induction on $m$, where  $m$ is the smallest index with $1\le m\le t_1+\cdots+t_r$ such that $x_m$ is not in the form required by Theorem \ref{SSM}. Then we have an expression $m=t_1+\cdots+t_{i-1}+j$ with $s_i<j\le t_i$. Thus (with the notation of Theorem 4.8 \cite{CP})
$$
x_{t_1+\cdots+t_{i-1}+j}=\delta_{i,j}y_{t_1+\cdots+t_{i-1}+j}y_{t_1+\cdots+t_i+1}^{h_{j,t_1+\cdots+t_i+1}(i)}\cdots y_{t_1+\cdots+t_{r-1}+s_r}^{h_{j,t_1+\cdots+t_{r-1}+s_r}(i)}
$$
with all $h_{j,k}(i)$ non negative integers, and $h_{j,k}(i)=0$ if $k\ne t_1+\cdots+t_{u-1}+v$ with $1\le v\le s_u$ for some $u$.  The rank of the $s_u\times s_u$ matrix $(g_{\alpha,\beta}(u))$ is $s_u$ for $1\le u\le r$. Thus each column of $(g_{\alpha,\beta}(u))$ is nonzero, and so there exists a sequence of monoidal transforms $S\rightarrow S'$ along $\nu^*$ of the type $S(i,j,u,l)$ with 
$i< u\le r$ and $1\le l\le s_u$ such that we have new regular parameters  $y_m'$ in $S'$ with
$$
y_m=\left\{\begin{array}{l}
y_{t_1+\cdots+t_{i-1}+j}'(y_{t_1+\cdots+t_i+1}')^{b_{t_1+\cdots+t_i+1}}\cdots (y_{t_1+\cdots+t_{r-1}+s_r}')^{b_{t_1+\cdots+t_{r-1}+s_r}}\\
\mbox{ if }m=t_1+\cdots+t_{i-1}+j\\
y_m'\mbox{ if }m\ne t_1+\cdots+t_{i-1}+j
\end{array}\right.
$$
such that $b_l$ are non negative integers with $b_{t_1+\cdots+t_{u-1}+v}=0$ if $s_u<v\le t_u$ for $i+1\le u\le r$  and
$$
x_m=\delta_{i,j}y_{t_1+\cdots+t_{i-1}+j}'M
$$
where 
$$
M=(y_{t_1+\cdots+t_i+1}')^{b_{t_1+\cdots+t_i+1}+h_{j,t_1+\cdots+t_i+1}(i)}\cdots
(y_{t+1+\cdots+t_{r-1}+s_r}')^{b_{t_1+\cdots+t_{r-1}+s_r}+h_{j,t_1+\cdots+t_{r-1}+s_r}(i)}
$$
if $m=t_1+\cdots+t_{i-1}+j$ and since each column of $(g_{l,m}(u))$ is nonzero,  we can choose the $b_{t_1+\cdots+t_{u-1}+v}$ so that there exist nonnegative integers $c_{t_1+\cdots+t_{u-1}+v}$ such that
$$
\prod(x_{t_1+\cdots+t_{u-1}+v})^{c_{t_1+\cdots+t_{u-1}+v}}=\gamma M
$$
where $\gamma$ is a unit in $S$ and the product is over $i\le u-1\le r-1$ and $1\le v\le s_u$.  
Further, the expression of $x_m$ in terms of the $y_l'$ is obtained from the expression of $x_m$ in terms of $y_l$ by replacing $y_l$ with $y_l'$ for $1\le l\le t_1+\cdots+t_r$ if $m\ne t_1+\cdots+t_{i-1}+j$.

Now we have a monoidal transform $R\rightarrow R'$ along $\nu$  which factors through $S'$ defined by
$$
x_m=\left\{\begin{array}{ll}
x_{t_1+\cdots+t_{i+1}+j}'\prod(x_{t_1+\cdots+t_{u-1}+v}')^{c_{t_1+\cdots+t_{u-1}+v}}&\mbox{ if } m=t_1+\cdots+t_{i-1}+j\\
x_m'&\mbox{ if } m\ne t_1+\cdots+t_{i-1}+j.
\end{array}\right.
$$
We obtain that 
$$
x_{t_1+\cdots+t_{i-1}+j}'=\gamma \delta_{i,j}y_{t_1+\cdots+t_{i-1}+j}'
$$
and if $m\ne t_1+\cdots+t_{i-1}+j$, then the expression for $x_m'$ is terms of the $y_l'$ is obtained by replacing the $y_l$ variables with  $y_l'$ in the expression for $x_m$ in terms of the $y_l$. Finally, we have new variables $y_l''$ in $S'$, obtained by letting
$$
y_l''=\left\{\begin{array}{ll}
\gamma\delta_{i,j}y_{t_1+\cdots+t_{i-1}+j}'
&\mbox{ if }l=t_1+\cdots+t_{i-1}+j\\
y_l'&\mbox{ if }l\ne t_1+\cdots+t_{i-1}+j,
\end{array}\right.
$$
proving the induction statement.

\end{proof}

We now summarize some results on toric rings from \cite{BG}.
Suppose that $M$ is a finitely generated submonoid (subsemigroup) of $\ZZ^n$ for some $n\ge 0$. Let
\begin{equation}\label{eqA12}
\tilde M_{\ZZ^n}=\{v\in \ZZ^n\mid mv\in M\mbox{ for some }m\in \ZZ_{>0}\}.
\end{equation}
We have that $\tilde M_{\ZZ^n}=(\RR_{\ge 0} M)\cap \ZZ^n$ (Proposition 2.2 \cite{BG}).

\begin{Proposition}\label{BG1}(Proposition 2.43 \cite{BG}) Suppose $v_1,\ldots,v_n\in \ZZ^n$ are linearly independent. Let
$$
{\rm par}(v_1,\ldots,v_n)=\{q_1v_1+\cdots+q_nv_n\mid 0\le q_i<1\mbox{ for }i=1,\ldots,n\}.
$$
Let $M=\ZZ_{\ge 0}v_1+\cdots+\ZZ_{\ge 0}v_n$ be the submonoid of $\ZZ^n$ generated by $v_1,\ldots,v_n$. Then
\begin{enumerate}
\item[a)] $\Lambda =\ZZ^n\cap {\rm par}(v_1,\ldots,v_n)$ is a system of generators of the $M$-module $\tilde M_{\ZZ^n}$; that is, $\tilde M_{\ZZ^n}=\cup_{x\in \Lambda}(x+M)$.
\item[b)] $(a+M)\cap(b+M)=\emptyset$ for $a,b\in \Lambda$ with $a\ne b$.
\item[c)] $|\Lambda|=[\QQ U\cap \ZZ^n:U]$ where $U$ is the sublattice of $\ZZ^n$ generated by $M$.
\end{enumerate}
\end{Proposition}

\begin{Lemma}\label{BG2} (Lemma 4.40 \cite{BG}) Suppose that $k$ is a ring, $M$ is a finitely generated submonoid of $\ZZ^n$ and $k[z_1,\ldots,z_n]$ is a polynomial ring over $k$. Then $k[z^v\mid v\in \tilde M_{\ZZ^n}]$ is the integral closure of $k[v\mid v\in M]$ in $k(z_1,\ldots, z_n)$.
\end{Lemma}

Suppose that $R_0\rightarrow S$ satisfies the conclusions of Theorem \ref{SSM} and that $R_0/m_{R_0}\cong S/m_S$. Then after replacing the $x_i$ with the product of $x_i$ times an appropriate unit in $R_0$, we may assume that $\delta_{i,j}\equiv 1\mbox{ mod }m_S$ for all $i,j$.  Let $A=(a_{ij})$ where 
\begin{equation}\label{eqA51}
x_i=\delta_i\prod_{j=1}^ny_j^{a_{ij}}
\end{equation}
 with $\delta_i\in S$ units satisfying $\delta_i\equiv 1\mbox{ mod }m_S$. As in Theorem 4.2 \cite{CP}, from the adjoint matrix of $A$, for all $i=t_1+\cdots+t_{a-1}+b$ with $1\le a\le r$ and $1\le b\le s_a$, we have expressions
\begin{equation}\label{eqA1}
f_i=\prod x_j^{b_{ij}}=(\prod\delta_j^{b_{ij}})y_i^e
\end{equation}
where  the products are over $j=t_1+\cdots+t_{u-1}+v$ with $1\le u\le r$ and $1\le v\le s_u$, all $b_{ij}\in \ZZ$ and $e=|\mbox{Det}(A)|$. Let 
$$
R=\overline{R_0[\Omega]}_{m_{\nu}\cap \overline{R_0[\Omega]}},
$$
where, 
$$
\Omega=\{f_i\mid i=t_1+\cdots+t_{u-1}+v\mbox{ for }1\le u\le r, 1\le v\le s_u\}
$$
and $\overline{R_0[\Omega]}$ denotes the integral closure of $R_0[\Omega]$ in $K$. Then $S$ dominates $R$ and $\sqrt{m_RS}=m_S$.
Let 
$$
\Theta=\{x_i\mid x_i=x_{t_1+\cdots+t_{u-1}+v}\mid 1\le u\le r, 1\le v\le s_u\}.
$$
We have that
$m_{\nu}\cap R_0[\Omega]=(x_1,\ldots,x_n,\Omega)$ and $k[x_1,\ldots,x_n]_{(x_1,\ldots,x_n)}\rightarrow R_0$ is unramifed, so 
$$
k[x_1,\ldots,x_n,\Omega]_{(x_1,\ldots,x_n,\Omega)}\rightarrow R_0[\Omega]_{m_{\nu}\cap R_0[\Omega]}
$$
is unramified as $(x_1,\ldots,x_n,\Omega)R_0[\Omega]_{m_{\nu}\cap R_0[\Omega]}$ is the maximal ideal of $R_0[\Omega]_{m_{\nu}\cap R_0[\Omega]}$. There exist Laurent monomials $M_1,\ldots, M_{\alpha}$ in the variables in $\Theta$ such that letting $\Sigma=\{M_1,\ldots, M_{\alpha}\}$, $k[x_1,\ldots,x_n,\Omega,\Sigma]$ is the integral closure of $k[x_1,\ldots,x_n,\Omega]$ in $k(x_1,\ldots,x_n)$ by Lemma  \ref{BG2}.

We have that $\nu(M_i)$ is positive for $1\le i\le \alpha$, since some power of $M_i$ is a monomial in elements of $\Omega$ and $\Theta$ by (\ref{eqA12}). Thus $m_{\nu}\cap R_0[\Omega,\Sigma] = (x_1,\ldots,x_n,\Omega,\Sigma)$. Now 
\begin{equation}\label{eqA21}
k[x_1,\ldots,x_n,\Omega,\Sigma]_{(x_1,\ldots,x_n,\Omega,\Sigma)}\rightarrow R_0[\Omega,\Sigma]_{(x_1,\ldots,x_n,\Omega,\Sigma)}
\end{equation}
is unramified, and thus by  Corollary 9.11, Expos\'e I \cite{SGAI},  $R_0[\Omega,\Sigma]_{(x_1,\ldots,x_n,\Omega,\Sigma)}$ is normal.
The elements of $\Omega$ and $\Sigma$ are Laurent monomials in the variables $\Theta$.

Recall that 
\begin{equation}\label{eqA3}
\{\nu(x_i)\mid x_i\in \Theta\}\mbox{ is a basis of the $\QQ$-vector space  }\Phi_{\nu}\otimes_{\ZZ}\QQ.
\end{equation}
Now (Proposition on page 48 \cite{F}) there exists a toric resolution of singularities of 

\noindent $\mbox{Spec}(k[x_1,\ldots,x_n,\Omega,\Sigma])$, so there exist Laurent monomials $\overline N=\{N_1,\ldots, N_{\beta}\}$ in the variables of $\Theta$ such that 
\begin{equation}\label{eqA4}
\nu(N_j)\ge 0\mbox{ for all }j,
\end{equation}
and $k[x_1,\ldots,x_n,\Omega,\Sigma, \overline N]$ is a regular  ring. By (\ref{eqA3}) and (\ref{eqA4}), we have that $\nu(N_j)>0$ for all $j$, so that 
$$
k[x_1,\ldots,x_n,\Omega,\Sigma,\overline N]\cap m_{\nu}=(x_1,\ldots,x_n, \Omega, \Sigma,\overline N).
$$
Thus $k[x_1,\ldots,x_n,\Omega,\Sigma,\overline N]_{(x_1,\ldots,x_n,\Omega, \Sigma,\overline N)}$ is a regular local ring which is dominated by $\nu$. Thus we have that
$$
k[x_1,\ldots,x_n,\Omega,\Sigma,\overline N]_{(x_1,\ldots,x_n,\Omega,\Sigma,\overline N)}\rightarrow
R_0[\Omega,\Sigma,\overline N]_{(x_1,\ldots,x_n,\Omega,\Sigma,\overline N)}
$$
is unramified and
$$
R_0[\Omega,\Sigma,\overline N]_{(x_1,\ldots,x_n,\Omega,\Sigma,\overline N)}
$$
is a regular local ring which is dominated by $\nu$. 

Now each $N_j=\gamma_j\tilde N_j$ where $\gamma_j\in S$ is a unit and $\tilde N_j$ is a Laurent monomial in the variables
$$
\tilde \Theta=\{y_i\mid y_i=y_{t_1+\cdots+t_{u-1}+v}\mid 1\le u\le r, 1\le v\le s_u\}.
$$
 Let $\tilde N=\{\tilde N_1,\ldots,\tilde N_{\beta}\}$.
Then $S[\overline N]_{(y_1,\ldots,y_n,\overline N)}=S[\tilde N]_{(y_1,\ldots,y_n,\tilde N)}$ is dominated by $\nu^*$. By Lemma \ref{BG2}, by adding finitely many more Laurent monomials $\tilde P=\{\tilde P_1,\ldots,\tilde P_{\gamma}\}$ in the variables $\tilde \Theta$, we have that 
$k[y_1,\ldots,y_n,\tilde N,\tilde P]$ is normal and is dominated by $\nu^*$, 
$k[y_1,\ldots,y_n,\tilde N]\rightarrow k[y_1,\ldots,y_n,\tilde N,\tilde P]$ is finite and 
$$
k[y_1,\ldots, y_n,\tilde N,\tilde P]_{(y_1,\ldots,y_n,\tilde N,\tilde P)}\rightarrow S[\tilde N,\tilde P]_{(y_1,\ldots,y_n,\tilde N,\tilde P)}
$$
is unramified.  Further,
$S[\tilde N,\tilde P]_{(y_1,\ldots,y_n,\tilde N,\tilde P)}$ is normal by Corollary  9.11, Expos\'e I \cite{SGAI}, lies over $S[\tilde N]_{(y_1,\ldots,y_n,\tilde N)}$ and is dominated by $\nu^*$. Thus $S[\tilde N,\tilde P]_{(y_1,\ldots,y_n,\tilde N,\tilde P)}$ is the normal local ring of $K^*$ which is dominated by $\nu^*$ and lies over $R_0[\Omega,\Sigma, \overline N]_{(x_1,\ldots,x_n,\Omega,\Sigma,\overline N)}$.

Let $C=k[x_1,\ldots,x_n,\Omega,\Sigma,\overline N]$. The $k$-algebra $C$ is $\ZZ^n$-graded by the grading $\deg x_i=e_i\in \ZZ^n$ for $1\le i\le n$, where $e_i$ is the vector with a 1 in the $i$-th place and zeros everywhere else. Let $m=(x_1,\ldots,x_n,\Omega,\Sigma,\overline N)$ be the $\ZZ^n$-graded maximal ideal of $C$. Since $C$ is an $n$-dimensional regular  ring, $\dim_k m/m^2=n$. Let $T_1,\ldots,T_n\in \{x_1,\ldots,x_n\}\cup\Omega\cup\Sigma\cup \overline N$ form a $k$-basis of $m/m^2$. A nonzero homogeneous element $U$ of $C$ (with respect to the 
$\ZZ^n$-grading)  is uniquely determined up to multiplication by a nonzero element of $k$. Thus for $l\in\NN$, a nonzero homogeneous element $U$ of $m^l\setminus m^{l+1}$ is a monomial in $T_1,\ldots,T_n$ times a nonzero element of $k$. Thus $C=k[T_1,\ldots,T_n]$. Since $\dim C=n$, $T_1,\ldots,T_n$ are algebraically independent over $k$, and so $C$ is a polynomial ring over $k$ in the variables $T_1,\ldots,T_n$.

Since $\delta_i\equiv 1\mbox{ mod }m_S$ for all $i$ in (\ref{eqA51}), and since a complete local ring is Henselian, and $k$ has characteristic zero, there exist units $\epsilon_i\in \hat S$ for $1\le i\le n$ such that  $\epsilon_i=1$ if $y_i\not\in \tilde\Theta$, and setting $z_i=\epsilon_iy_i$,
\begin{equation}\label{eqA2}
x_i=\prod z_j^{a_{ij}}\mbox{ for }x_i\in \Theta\mbox{ and }x_i=y_i=z_i\mbox{ for }x_i\not\in \Theta
\end{equation}
where the product is over $z_j$ with $y_j\in \tilde\Theta$.
Let
\begin{equation}\label{eqA22}
R_1=R_0[\Omega,\Sigma,\overline N]_{(x_1,\ldots,x_n,\Omega,\Sigma,\overline N)}\mbox{ and }
S_1=S[\tilde N,\tilde P]_{(y_1,\ldots,y_n,\tilde N,\tilde P)}.
\end{equation}
Define Laurent monomials $\hat N$ and $\hat P$ in the variables $\{z_i\mid y_i\in \tilde \Theta\}$ by replacing the variables $y_i$ in the monomials in $\tilde N$ and $\tilde P$ with the $z_i$. Then $k[z_1,\ldots,z_n,\hat N,\hat P]$ is normal, since it is normal when we replace the $z_i$ variables with the $y_i$ variables. We have a commutative diagram of $k$-algebra homomorphisms 
$$
\begin{array}{ccc}
R_0&\rightarrow &\hat S\\
\uparrow&&\uparrow\\
k[x_1,\ldots,x_n]&\rightarrow& k[z_1,\ldots,z_n]
\end{array}
$$
where the map $k[z_1,\ldots,z_n]\rightarrow \hat S$ is defined  by $z_i\mapsto \epsilon_iy_i$ for $1\le i\le n$ and the map $k[x_1,\ldots,x_n]\rightarrow k[z_1,\ldots,z_n]$ is defined by $x_i\mapsto \prod_jz_j^{a_{ij}}$for $1\le i\le n$. We have that the induced homomorphisms $k[x_1,\ldots,x_n]_{(x_1,\ldots,x_n)}\rightarrow R_0$ and $k[z_1,\ldots,z_n]_{(z_1,\ldots,z_n)}\rightarrow \hat S$ are unramified. 

We have an induced commutative diagram of $k$-algebra homomorphisms
$$
\begin{array}{ccc}
R_1&\rightarrow &\hat S_1\\
\uparrow&&\uparrow\\
k[T_1,\ldots,T_n]&\rightarrow &k[z_1,\ldots,z_n,\hat N,\hat P]
\end{array}
$$
where the induced homomorphisms $k[T_1,\ldots,T_n]_{(T_1,\ldots,T_n)}\rightarrow R_1$ and 
$$
k[z_1,\ldots,z_n,\hat N,\hat P]_{(z_1,\ldots,z_n,\hat N,\hat P)}\rightarrow \hat S_1
$$
 are unramified and the $T_j$ are Laurent monomials in the $z_i$ variables.

We have that $k[z_1,\ldots,z_n,\hat N,\hat P]$ is finite over $k[z_1,\ldots,z_n,\hat N]$ since $k[y_1,\ldots,y_n,\tilde N,\tilde P]$ is finite over $k[y_1,\ldots,y_n,\tilde N]$. The $k$-algebra $k[z_1,\ldots,z_n]$ is finite over $k[x_1,\ldots,x_n,\Omega]$ since  $f_i=z_i^e$ by (\ref{eqA1}) and by (\ref{eqA2}). Thus $k[z_1,\ldots,z_n,\hat N,\hat P]$ is finite over $k[T_1,\ldots,T_n]=k[x_1,\ldots,x_n,\Omega,\Sigma,\overline N]$ and so $k[z_1,\ldots,z_n,\hat N,\hat P]$ is the integral closure of $k[T_1,\ldots,T_n]$ in $k(z_1,\ldots,z_n)$. Let $D=k[z_1,\ldots,z_n,\hat N,\hat P]$. Express
$T_i=\prod z_j^{e_{ij}}$ for $1\le i\le n$. Let  $E=(e_{i,j})$.

Recall the notation introduced before the statement of Proposition \ref{BG1}. Let $M$ be the submonoid of $\ZZ^n$ generated by
$$
\{(e_{i,1},\ldots,e_{i,n})\mid 1\le i\le n\}.
$$
Then $D=k[z^v\mid v\in \tilde M_{\ZZ}]$ by Lemma \ref{BG2}. Now $k[x_1,\ldots,x_n]\rightarrow k[T_1,\ldots,T_n]$ is birational and the $T_i$ are Laurent monomials in $x_1,\ldots,x_n$, so there exists an $n\times n$ matrix $Q=(q_{ij})$ such that $x_i=\prod_jT_j^{q_{ij}}$ with $|{\rm Det}(Q)|=1$. Thus $A=QE$, and so $|{\rm Det}(E)|=|{\rm Det}(A)|$. We have
\begin{equation}\label{eqA5}
|{\rm Det}(E)|=[\ZZ^n:E\ZZ^n]=[\ZZ^n:A\ZZ^n]=|{\rm Det}(A)|.
\end{equation}
 By Proposition \ref{BG1}, there exists a subset $\Lambda$ of $\ZZ^n$ such that $D$ is a free $k[T_1,\ldots,T_n]$-module with $k[T_1,\ldots,T_n]$-basis $\{z^{\sigma}\mid \sigma\in \Lambda\}$ and $|\Lambda|=[\ZZ^n:E\ZZ^n]$. We have that $z^{\sigma}=\lambda_{\sigma}y^{\sigma}$ where $\lambda_{\sigma}\in \hat S_1$  are units. $y^{\sigma}=\lambda_{\sigma}^{-1}z^{\sigma}\in \hat S_1\cap K^*=S_1$ by Lemma 2 \cite{LU}. 
 There exist $g_1,\ldots,g_c\in R_1$ such that $\{\overline g_1,\ldots,\overline g_c\}$ is a $k$-basis of $S_1/m_{S_1}=R_1/m_{R_1}$, where $\overline g_i$ is the class of $g_i$ in $R_1/m_{R_1}$. Now $D_{(z_1,\ldots,z_n,\hat N,\hat P)}\rightarrow \hat S_1$ is unramified, so $m_D\hat S_1=m_{\hat S_1}$, where $m_D=(z_1,\ldots,z_n,\hat N,\hat P)$ and so 
 $$
 \hat S_1=\sum_{i=1}^cg_i\hat D+m_D\hat S_1.
 $$
 We have that $\hat S_1$ is a finite $\hat D$-module, so $\hat S_1=\sum_{i=1}^cg_i\hat D$ by Nakayama's lemma. Hence 
 $$
 \hat S_1=\sum_{i=1}^c\sum_{\sigma \in \Lambda}g_i(\lambda_{\sigma}y^{\sigma})\hat D
 =\sum_{\sigma\in\Lambda}(\lambda_{\sigma}y^{\sigma})\hat R_1.
 $$
 Now $S_1$ lies over $R_1$, so there exists $t$ such that $m_{S_1}^t\subset m_{R_1}S_1$ which implies $m_{\hat S_1}^t\subset m_{\hat R_1}\hat S_1$. Thus  there exist $\tau_{\sigma}\in S_1$ such that $\tau_{\sigma}\equiv \lambda_{\sigma}\mbox{ mod }m_{\hat S_1}^t$, so that $\hat S_1=\sum_{\sigma\in \Lambda}(\tau_{\sigma}y^{\sigma})\hat R_1+m_{\hat R_1}\hat S_1$. This  implies $\hat S_1=\sum(\tau_{\sigma}y^{\sigma})\hat R_1$ by Nakayama's lemma, since $\hat S_1$ is a finitely generated $\hat R_1$-module. Now $S_1/m_{S_1}^t\cong \hat S_1/m_{S_1}^t\hat S_1$ and $m_{S_1}^t\subset m_{R_1}S_1$ which implies $S_1=\sum_{\sigma\in \Lambda}(\tau_{\sigma}y^{\sigma})R_1+m_{R_1}S_1$. If $R_1\rightarrow S_1$ is finite, we then have 
 \begin{equation}\label{eqA52}
 S_1=\sum_{\sigma\in \Lambda}(\tau_{\sigma}y^{\sigma})R_1
 \end{equation}
  by Nakayama's lemma.

\begin{Proposition}\label{PropVG} With the above notation, assume that $R_1\rightarrow S_1$ is finite,
\begin{equation}\label{eqA6}
e=[\Phi_{\nu^*}/\Phi_{\nu}]=|{\rm Det}(A)|\mbox{ and }
\end{equation}
\begin{equation}\label{eqA7}
\ZZ^n/A^t\ZZ^n\cong \Phi_{\nu^*}/\Phi_{\nu}
\end{equation}
by the map $(b_1,\ldots,b_n)\mapsto b_1\nu^*(y_1)+\cdots+b_n\nu^*(y_n)$. Then
$$
\{\nu^*(\tau_{\sigma} y^{\sigma})|\sigma\in \Lambda\}
$$
is a complete set of representatives of the cosets of $\Phi_{\nu}$  in $\Phi_{\nu^*}$.
\end{Proposition}

\begin{proof} Assume that $\sigma_1,\sigma_2\in \Lambda$ and $\nu^*(\tau_{\sigma_1}y^{\sigma_1})-\nu^*(\tau_{\sigma_2}y^{\sigma_2})\in \Phi_{\nu}$. Then $\nu^*(y^{\sigma_1-\sigma_2})\in \Phi_{\nu}$. By (\ref{eqA7}), $y^{\sigma_1-\sigma_2}=ux^{\lambda}$ for a suitable unit $u\in S$ and $\lambda\in \ZZ^n$, and so $z^{\sigma_1-\sigma_2}=x^{\lambda}$.
Thus, writing $\lambda$ as a difference of elements of $M$,  $\sigma_1=\sigma_2$ by Proposition \ref{BG1} b).  The proposition now follows since 
$$
|\Lambda|=|{\rm Det}(E)|=|{\rm Det}(A)|=e
$$
by (\ref{eqA5}) and (\ref{eqA6}).
\end{proof}

We now give the proof of Theorem \ref{ThmAJV}. Let notation be as in the statement of the theorem. Let $\overline K$ be a Galois closure of $K\rightarrow K^*$, and $\overline\nu$ be an extension of $\nu^*$ to $\overline K$. Let $K^s=\overline K^{G^s(V_{\overline\nu}/V_{\nu})}$, $K^i=\overline K^{G^i(V_{\overline\nu})/V_{\nu})}$ and $(K^*)^s=\overline K^{G^s(V_{\overline\nu}/V_{\nu^*})}$. We have that $(K^*)^s=K^*K^s$ since
$$
G(\overline K/K^*K^s)=G(\overline K/K^*)\cap G^s(V_{\overline \nu}/V_{\nu})=G^s(V_{\overline \nu}/V_{\nu^*}).
$$
Let $K'=K^i\cap (K^*)^s$ and $\nu'=\overline\nu|K'$. Let $(\nu^*)^s=\overline\nu|(K^*)^s$.
We have a commutative diagram of field extensions
$$
\begin{array}{ccccccc}
&&K^s&\rightarrow&K'&\rightarrow&K^i\\
&\nearrow&&\searrow&\downarrow&&\downarrow\\
K&&&&(K^*)^s&\rightarrow & \overline K.\\
&\searrow&&\nearrow\\
&&K^*.
\end{array}
$$
Let $H=G(\overline K/K')$, which is the subgroup of $G(\overline K/K)$ which is generated by $G^i(V_{\overline \nu}/V_{\nu})$ and $G^s(V_{\overline\nu}/V_{\nu^*})$.
If $\sigma\in H$, we have that $\sigma:V_{\overline\nu}\rightarrow V_{\overline \nu}$. Thus $V_{\nu'}=V_{\overline\nu}^H$.
We have a commutative diagram
$$
\begin{array}{ccc}
V_{\overline\nu}&\stackrel{\sigma}{\rightarrow}&V_{\overline\nu}\\
\downarrow&&\downarrow\\
V_{\overline\nu}/m_{\overline\nu}&\stackrel{\overline\sigma}{\rightarrow}&V_{\overline\nu}/m_{\overline\nu}
\end{array}
$$
for $\sigma\in H$ where $\overline\sigma$ is the induced homomorphism, so $H$ acts on $V_{\overline\nu}/m_{\overline\nu}$, and $V_{\nu'}/m_{\nu'}=(V_{\overline\nu}/m_{\overline\nu})^H$. 
Now $\sigma\in G^i(V_{\overline\nu}/V_{\nu})$ implies $\overline\sigma:V_{\overline\nu}/m_{\overline\nu}\rightarrow V_{\overline\nu}/m_{\overline\nu}$ is the identity map (by Theorem 1.48 \cite{RTM}). Thus
$$
(V_{\overline\nu}/m_{\overline\nu})^H=(V_{\overline\nu}/m_{\overline\nu})^{G^s(V_{\overline\nu}/V_{\nu^*})}
=V_{(\nu^*)^s}/m_{(\nu^*)^s}=V_{\nu^*}/m_{\nu^*}
$$
by Theorems 1.47 and 1.48 \cite{RTM}. Thus
\begin{equation}\label{eqA55}
V_{\nu'}/m_{\nu'}=V_{(\nu^*)^s}/m_{(\nu^*)^s}=V_{\nu^*}/m_{\nu^*}.
\end{equation}

We have $\Phi_{\overline \nu | K'}=\Phi_{\nu}$ and $\Phi_{\overline \nu |(K^*)^s}=\Phi_{\nu^*}$ (by Theorem 23, page 71 \cite{ZS2} or  Proposition \ref{PropSplit} and Proposition \ref{PropInert}). By 
Lemmas \ref{Lemma8}, \ref{Lemma4} and Proposition \ref{PropSplit},
 there exists a normal algebraic local ring $R_2$ of $K$ which is dominated by $\nu$ and dominates $\tilde R$ such that if $R$ is a normal algebraic local ring of $K$ which dominates $R_2$, then letting $S$ be the normal local ring of $K^*$ which is dominated by $\nu^*$ and lies over $R$ and $\overline R$ be the normal local ring of $\overline K$ which is dominated by $\overline\nu$ and lies over $R$, we have that $G^s(\overline R/R)=G^s(V_{\overline\nu}/V_{\nu})$, $G^s(\overline R/S)=G^s(V_{\overline\nu}/V_{\nu^*})$, 
\begin{equation}\label{eqA23}
{\rm gr}_{\nu^*}(S)\cong{\rm gr}_{\overline \nu}(S^s)\mbox{ and }{\rm gr}_{\nu}(R)\cong {\rm gr}_{\overline\nu}(R^s)
\end{equation}
 where $S^s$ is the local ring of $(K^*)^s$ which lies over $S$ and $R^s$ is the local ring of $K^s$ which lies over $R$. Further, $R\rightarrow R^s$ and $S\rightarrow S^s$ are unramified. 
 
 Letting $R'$ be the normal local ring of $K'$ which is dominated by $\nu'$ and lies over $R$, we may assume by Lemmas \ref{Lemma8}, \ref{Lemma9} and \ref{Lemma4} that $R_2$ is such that (for $R$ which dominates $R_2$)
 $$
 G^i(\overline R/R')=G^i(V_{\overline\nu}/V_{\nu'})
 $$
 and
 $$
 G^i(\overline R/S^s)=G^i(V_{\overline\nu}/V_{(\nu^*)^s}).
 $$
 Thus $[V_{\overline\nu}/m_{\overline \nu}:V_{\nu'}/m_{\nu'}]=[\overline R/m_{\overline R}:R'/m_{R'}]$ and
$$
 [V_{\overline \nu}/m_{\overline\nu}:V_{(\nu^*)^s}/m_{(\nu^*)^s}]=[\overline R/m_{\overline R}:S^s/m_{S^s}]
 $$
  by Theorem 1.48 \cite{RTM}, and so by (\ref{eqA55}),
 \begin{equation}\label{eqA53}
 [S^s/m_{S^s}:R'/m_{R'}]=[V_{(\nu^*)^s}/m_{(\nu^*)^s}:V_{\nu'}/m_{\nu'}]=1.
 \end{equation}
 We have that ${\rm gr}_{\nu'}(R')\cong{\rm gr}_{\nu}(R)\otimes_{R/m_R}S/m_S$ by Proposition \ref{PropUnR}.

By Theorem 6.1  \cite{CP} and Theorem 4.10 \cite{CP}, there exists a normal algebraic local ring $F_1$ of $K^s$ such that if $F\rightarrow G$ is a dominant map of regular algebraic local rings of $K^s$ and $(K^*)^s$ respectively such that $F$ dominates $F_1$ and $F\rightarrow G$ satisfies the conclusions of Theorem \ref{SSM}, then 
\begin{equation}\label{eqA8}
|{\rm Det}(A)|=e=[\Phi_{\nu^*}/\Phi_{\nu}]\mbox{ and }
\end{equation}
\begin{equation}\label{eqA9}
\ZZ^n/A^t\ZZ^n\cong \Phi_{\nu^*}/\Phi_{\nu}
\end{equation}
by the map $(b_1,\ldots,b_n)\mapsto b_1\nu^*(y_1)+\cdots+b_n\nu^*(y_n)$, where $A$ is the matrix of exponents of the conclusions of Theorem \ref{SSM}.

By Lemma \ref{Lemma4}, there exists a normal algebraic local ring $R_3$ of $K$ which is dominated by $\nu$, such that $R_3$ dominates $R_2$ and if $R$ is a normal algebraic local ring of $K$ which dominates $R_3$ and is dominated by $\nu$ and $R^s$ is the local ring of $K^s$ which is dominated by $\overline\nu$ and lies over $R$, then $R^s$ dominates $F_1$.

Now suppose that $R_0$ is an algebraic regular local ring of $K$ which dominates $R_3$ and is dominated by $\nu$, and that $R_0\rightarrow S$ satisfies the conclusions of Theorem \ref{SSM}, where $S$ is a local ring of $K^*$. We then construct a commutative diagram
$$
\begin{array}{ccc}
R_1&\rightarrow&S_1\\
\uparrow&&\uparrow\\
R_0&\rightarrow&S
\end{array}
$$
where $R_1$ and $S_1$ are defined by (\ref{eqA22}). Let
$$
\begin{array}{ccc}
R_1'&\rightarrow&S_1^s\\
\uparrow&&\uparrow\\
R_0'&\rightarrow&S^s
\end{array}
$$
be the sequence of normal algebraic local rings in $K'$, respectively $(K^*)^s$ lying over these rings. We have that $R_0\rightarrow R_0'$, $R_1\rightarrow R_1'$, $S\rightarrow S^s$ and $S_1\rightarrow S_1^s$ are unramified by Proposition \ref{PropUnR}, and by (\ref{eqA53}), 
\begin{equation}\label{eqA54}
S^s/m_{S^s}=R_0'/m_{R_0'}\mbox{ and }S_1^s/m_{S_1^s}=R_1'/m_{R_1'}.
\end{equation}
Now $R_1'\rightarrow S_1^s$ is finite since $S_1^s$ is the unique local ring of $(K^*)^s$ which lies over $R_1'$ by Proposition 1.46 \cite{RTM}, as $G^s(V_{\overline \nu}/V_{\nu})=G^s(\overline{R_1}/R_1)$ where $\overline{R_1}$ is the local ring of $\overline K$ which dominates $R_1$ and is dominated by $\overline\nu$. Thus
$$
S_1^s=\sum_{\sigma\in\Lambda}(\tau_{\sigma}y^{\sigma})R_1'
$$
by (\ref{eqA52}) and (\ref{eqA54}),  and $\{\tau_{\sigma}y^{\sigma}\}_{\sigma\in \Lambda}$ is a complete set of representatives of $\Phi_{\nu^*}/\Phi_{\nu}$ by (\ref{eqA8}), (\ref{eqA9}) and Proposition \ref{PropVG}.
In particular, if $f\in S_1^s$, then we have an expression 
\begin{equation}\label{eqA40}
f=\sum g_{\sigma}(\tau_{\sigma}y^{\sigma})
\end{equation}
with $g_{\sigma}\in R_1'$ and  $\overline \nu(g_{\sigma}\tau_{\sigma}y^{\sigma})$ are all distinct for the terms with $g_{\sigma}\ne 0$. Thus
\begin{equation}\label{eqA11}
\overline\nu(f)=\min\{\overline\nu(g_{\sigma})+\overline\nu(\tau_{\sigma}y^{\sigma})\mid \sigma\in \Lambda\},
\end{equation}
and  the expression (\ref{eqA40}) is unique. Thus
\begin{equation}\label{eqA10}
S_1^s\cong \bigoplus_{\sigma\in\Lambda}R_1'\tau_{\sigma}y^{\sigma}.
\end{equation}

Thus, since ${\rm in}_{\overline \nu}(\tau_{\sigma})\in S_1/m_{S_1}\cong R_1'/m_{R_1'}$ for all $\sigma$, 
$$
{\rm gr}_{\nu^*}(S_1)\cong {\rm gr}_{\overline \nu}(S_1^s)\cong \oplus_{\sigma\in \Lambda}{\rm gr}_{\overline \nu}(R_1'){\rm in}_{\overline \nu}(y^{\sigma})\cong \oplus_{\sigma\in \Lambda}({\rm gr}_{\nu}(R_1)\otimes_{R_1/m_{R_1}}S_1/m_{S_1}){\rm in }_{\nu^*}(y^{\sigma})
$$
by (\ref{eqA23}) and Proposition \ref{PropUnR}, establishing Theorem \ref{ThmAJV}.

\vskip .2truein
We now prove Proposition \ref{PropAJV2}. Let notation and assumptions be as in the statement of Proposition \ref{PropAJV2}, and continue with the notation of the proof of Theorem \ref{ThmAJV}. Let $L={\rm QF}(R_1^s)=K^s$ and $M={\rm QF}(S_1^s)=(K^*)^s$. Since $k'\cong V_{\nu}/m_{\nu}$ is algebraically closed, $G^i(V_{\overline \nu}/V_{\nu^*})=G^s(V_{\overline \nu}/V_{\nu^*})$ and $G^i(V_{\overline \nu}/V_{\nu})=G^s(V_{\overline \nu}/V_{\nu})$. 
Since $k'$ has characteristic zero, the corollary on page 77 \cite{ZS2} implies that 
$G^i(V_{\overline \nu}/V_{\nu})\cong \Phi_{\overline\nu}/\Phi_{\nu}$, which is Abelian, so $(K^*)^i$ is Galois over $K^i$. Now $G((K^*)^i/K^i)=G^i((K^*)^i/K^i)$, so
$$
G(M/L)=G((K^*)^i/K^i)\cong \Phi_{\overline \nu|(K^*)^i}/\Phi_{\overline\nu|K^i}\cong \Phi_{\nu^*}/\Phi_{\nu},
$$
again by the corollary on page 77 \cite{ZS2} and Theorem 23, page 71 \cite{ZS2} or Proposition \ref{PropSplit}. Let $G=G(M/L)$. The ring $S_1^s$ is the integral closure of $R_1^s$ in $M$ so $G$ acts on $S_1^s$ and $(S_1^s)^G=R_1^s$ since $R_1^s$ is normal. 
Now $\overline \nu(\sigma(f))=\overline \nu(f)$ for all $\sigma\in G$ and $f\in S_1^s$ by (3) on page 68 \cite{ZS2} (since $G=G^s(M/L)$) so $G$ acts on ${\rm gr}_{\overline \nu}(S_1^s)$. We will now show that $({\rm gr}_{\overline \nu}(S_1^s))^G={\rm gr}_{\overline \nu}(R_1^s)$.

Suppose $u\in S_1$ and $\sigma({\rm in}_{\overline \nu}(u))={\rm in }_{\overline \nu}(u)$ for all $\sigma\in G$. Then there exist $h_{\sigma}\in S_1^s$ such that $\sigma(u)=u+h_{\sigma}$ with $\overline \nu(h_{\sigma})>\overline \nu(u)$ for all $\sigma\in G$, and so there exist $f_{\sigma}\in m_{(\nu^*)^s}$ such that $\sigma(u)=(1+f_{\sigma})u$ for all $\sigma\in G$. Let
$$
y={\rm Tr}_{M/L}(u)=(e+h)u
$$
with $h\in m_{(\nu^*)^s}$.  We have that $y$ is an element of $L=K^s$, and $e+h$ is a unit in $V_{\overline \nu}$ (since $k'$ has characteristic zero). Thus $\overline \nu(u)=\overline \nu(y)\in \Phi_{\nu}$. From the expression
$$
u=\sum g_{\sigma}(\tau_{\sigma}y^{\sigma})
$$
of (\ref{eqA40}), we have that
${\rm in}_{\overline \nu}(u)={\rm in}_{\overline \nu}(g_0)\in {\rm gr}_{\overline \nu}(R_1^s)$, completing the proof of Proposition \ref{PropAJV2},
since ${\rm gr}_{\overline\nu}(S_1^s)\cong{\rm gr}_{\nu^*}(S_1)$ and ${\rm gr}_{\overline\nu}(R_1^s)\cong {\rm gr}_{\nu}(R^s)$ by equation (\ref{eqA23}), as $R_1$ dominates $R_3$.

\end{document}